\documentclass[12pt]{article}
\usepackage{amsmath,amssymb,amsthm,amsfonts,mathrsfs}
\usepackage{appendix}
\usepackage{array,enumitem,graphics,xcolor,yfonts,colonequals}
\usepackage{stmaryrd}
\usepackage[all,cmtip]{xy}
\usepackage{tikz-cd}
\usepackage[colorlinks,anchorcolor=blue,citecolor=blue,linkcolor=blue,urlcolor=blue,bookmarksopen=true]{hyperref}
\usepackage{comment}
\usepackage{mathtools}
\usepackage[margin=1in]{geometry}

\usepackage{titlesec}
\titleformat{\section}
  {\normalfont\large\bf}{\thesection}{1em}{}
\titleformat{\subsection}[runin]
  {\normalfont\normalsize\bf}{\thesubsection}{1em}{}
\titleformat{\subsubsection}[runin]
  {\normalfont\normalsize\bf}{\thesubsubsection}{1em}{}

\usepackage{fancyhdr}
\fancypagestyle{titlepage}
{
	\fancyhf{}

	\fancyfoot[l]{
	\href{https://mathscinet.ams.org/mathscinet/msc/msc2020.html}
	}
}

\AtBeginDocument{%
	\def\MR#1{}
}

\newcommand{\bC}{\mathbb{C}}    
\newcommand{\bH}{\mathbb{H}}    
\newcommand{\bL}{\mathbb{L}}
\newcommand{\bP}{\mathbb{P}}    
\newcommand{\bR}{\mathbb{R}}    
\newcommand{\bZ}{\mathbb{Z}}    

\newcommand{\cA}{\mathcal{A}}   
\newcommand{\cD}{\mathcal{D}}
\newcommand{\cE}{\mathcal{E}}   
\newcommand{\cF}{\mathcal{F}}   
\newcommand{\cL}{\mathcal{L}}  
\newcommand{\cO}{\mathcal{O}}   
\newcommand{\cP}{\mathcal{P}}   
\newcommand{\cS}{\mathcal{S}}
\newcommand{\cT}{\mathcal{T}}
\newcommand{\cW}{\mathcal{W}}

\newcommand{\Alg}{\mathrm{Alg}}
\newcommand{\ch}{\mathrm{ch}}           
\newcommand{\Coh}{\mathrm{Coh}} 
\newcommand{\dd}{\mathrm{d}}
\newcommand{\Db}{\mathrm{D^b}}  
\newcommand{\DFuk}{\mathrm{DFuk}}

\newcommand{\Geo}{\mathrm{Geo}}
\newcommand{\GL}{\mathrm{GL}}           

\newcommand{\LG}{\mathrm{LG}}

\newcommand{\Mir}{\mathrm{Mir}}

\newcommand{\rank}{\mathrm{rank}}

\newcommand{\Stab}{\mathrm{Stab}} 
\newcommand{\Tw}{\mathrm{Tw}}

\newcommand{\Hom}{\operatorname{Hom}}   
\newcommand{\Perf}{\operatorname{Perf}}

\newcommand{\fs}{\mathfrak{s}}

\newtheorem*{thm*}{Theorem}
\newtheorem*{prop*}{Proposition}
\newtheorem*{cor*}{Corollary}
\newtheorem*{ques*}{Question}
\newtheorem{thm}{Theorem}[section]
\newtheorem{prop}[thm]{Proposition}
\newtheorem{cor}[thm]{Corollary}
\newtheorem{lemma}[thm]{Lemma}
\newtheorem{conj}[thm]{Conjecture}

\numberwithin{equation}{section}

\theoremstyle{definition}
\newtheorem{defn}[thm]{Definition}

\newtheorem{rmk}[thm]{Remark}

\begin{document}
\title{Special Lagrangians and Bridgeland stable objects beyond geometric stability conditions: the product case}
\author{Yu-Wei Fan}
\date{}
\maketitle
\thispagestyle{titlepage}

\begin{abstract}
We construct a family of non-geometric Bridgeland stability conditions on certain wrapped Fukaya categories, using homological mirror symmetry and categorical K\"unneth formulae. 
These stability conditions correspond to certain holomorphic volume forms, under which we prove that every stable object admits a special Lagrangian representative. 
This provides the first higher-dimensional examples of stability conditions away from the large complex structure limit for which ``stable implies special Lagrangian" is proved.
\end{abstract}

\setcounter{tocdepth}{2}
\tableofcontents

\section{Introduction.}

The visionary proposal of Thomas and Yau \cite{Thomas, ThomasYau} suggests a profound correspondence between the existence of \emph{special Lagrangian} submanifolds in Calabi--Yau manifolds and notions of \emph{stability} rooted in algebraic geometry. Building upon this framework, more recent developments have led to an extended version of these conjectures, using the language of Bridgeland stability conditions on derived Fukaya categories. This connection was briefly alluded to in the foundational work of Bridgeland \cite[Section~1.4]{BriStab}, where the notion of stability conditions on triangulated categories was first introduced, and was subsequently formulated more fully and explicitly by Joyce \cite{Joyce}. To provide the necessary context, we recall a simplified version of Joyce's conjecture, omitting its technical details.

\begin{conj}[{\cite[simplified version of Conjecture~3.2]{Joyce}}]\label{conj:Joyce}
Let $(M, \omega, \Omega)$ be a compact Calabi--Yau $n$-fold, where $\omega$ is a K\"ahler form and $\Omega$ is a holomorphic volume form. There exists a Bridgeland stability condition $\sigma=(Z,\cP)$ on the derived Fukaya category $\DFuk(M)$, whose objects are Lagrangian branes $\bL$, such that:
\begin{enumerate}[label=(\roman*)]
    \item The central charge $Z$ is given by the composition:
    $$
    K_0\left(\DFuk(M)\right)\xrightarrow{\bL\mapsto[L]}H_n(M,\bZ)\xrightarrow{[L]\mapsto\int_{[L]}\Omega}\bC,
    $$
    where $L\subseteq M$ denotes the underlying Lagrangian of $\bL$.
    \item An object $\bL$ is $\sigma$-semistable if and only if its underlying Lagrangian admits a (possibly singular) special Lagrangian representative.
\end{enumerate}
\end{conj}

\noindent This conjecture remains largely open. For a survey of recent progress in this direction, we refer the reader to \cite{Li_Yang}.

In the present work, we focus on the derived categories of toric Fano manifolds and their mirror Fukaya--Seidel-type categories. To fix the notation, let $X$ be a projective toric manifold of dimension $n$. Its mirror is expected to be a Landau--Ginzburg model $((\bC^\times)^n, W)$, where the superpotential $W\colon(\bC^\times)^n \to \bC$ is a Laurent polynomial. More concretely, let $V \subseteq \bZ^n$ be the set of primitive generators of the $1$-cones of the fan defining $X$, together with the origin. The mirror superpotential is then given by
$$
W_{\mathbf{a}}=\sum_{v\in V} a_v z^v, \quad z^v \coloneqq z_1^{v_1} \cdots z_n^{v_n}
$$
where $\mathbf{a} = \{a_v\}_{v \in V} \subseteq \bC$ is a set of generic coefficients. It is expected that the Fukaya--Seidel-type category associated with $W_{\mathbf{a}}$ is independent of a generic choice of $\mathbf{a}$, up to equivalence.

A folklore conjecture predicts that for generic coefficients $\mathbf{a}$, the category admits a Bridgeland stability condition whose central charge is given by integrating the holomorphic $n$-form
$$
\Omega_{\mathbf{a}} \coloneqq \exp\left(W_{\mathbf{a}}\right) \frac{\dd z_1}{z_1} \wedge \cdots \wedge \frac{\dd z_n}{z_n}
$$
along Lagrangian thimbles. Stable objects should then correspond to those that admit special Lagrangian representatives with respect to $\Omega_{\mathbf{a}}$. Note that central charges of this form have appeared previously in the literature, for instance in \cite{Iritani,Fang}.

So far, the conjecture has been verified only in dimension one, as a particular case of the seminal work of Haiden, Katzarkov, and Kontsevich \cite{HKK}. To state their result, we adopt the framework of partially wrapped Fukaya categories with stops \cite{GPS20,GPS} as our model for a Fukaya--Seidel-type category.
Let $\cW(\bC^\times,\fs)$ be the partially wrapped Fukaya category of $\bC^\times$ with stops at $2$ points on the boundary.
There is a homological mirror symmetry equivalence \cite[Example~1.22]{GPS}
$$
H^0\Tw(\cW(\bC^\times,\fs))\cong\Db(\bP^1).
$$

\begin{thm}[{\cite[Theorems 5.2 and 6.3]{HKK}}]\label{thm:intro-HKK}
There exists an isomorphism
$$
\bC^2_{(a,c)}\xrightarrow{\sim}\Stab H^0\Tw\left(\cW(\bC^\times,\fs)\right); \quad
(a,c)\mapsto\sigma_{a,c}=(Z_{a,c},\cP_{a,c})
$$
satisfying the following properties:
\begin{enumerate}[label=(\roman*)]
    \item The central charge $Z_{a,c}$ is given by the oscillatory integral
$$
Z_{a,c}(\bL)=\int_{[\bL]}\exp\left(z+c+\frac{e^a}{z}\right)\frac{\dd z}{z}.
$$
    \item An object $\bL$ is an indecomposable $\sigma_{a,c}$-semistable object of phase $\phi$ if and only if it admits a (graded) special Lagrangian representative of phase $\phi$ (equipped with an indecomposable local system) with respect to the holomorphic $1$-form
$$
\Omega_{a,c}\coloneqq\exp\left(z+c+\frac{e^a}{z}\right)\frac{\dd z}{z}.
$$
    In particular, every $\sigma_{a,c}$-stable object $\bL$ admits a special Lagrangian representative with respect to $\Omega_{a,c}$.
\end{enumerate}
\end{thm}

The main result of this paper is a higher-dimensional generalization of Theorem~\ref{thm:intro-HKK}. By the K\"unneth formula for wrapped Fukaya categories \cite[Corollary~1.18]{GPS}, there is an equivalence
$$
\cW(\bC^\times,\fs)^{\otimes n}\cong\cW((\bC^\times)^n,\fs_n),
$$
where $(\bC^\times)^n$ is equipped with the split symplectic form
$$
\omega_n\coloneqq\pi_1^*\omega_{\bC^\times}+\cdots+\pi_n^*\omega_{\bC^\times},
$$
with $\omega_{\bC^\times}$ denoting the canonical exact symplectic form on the cotangent bundle $T^*S^1\cong\bC^\times$, and $\fs_n$ denoting the \emph{product stop} \cite{GPS}. Our main result establishes the existence of a family of stability conditions on this category satisfying the property that every stable object is special Lagrangian.

\begin{thm}[see Theorem~\ref{thm:SG-main-thm}]\label{thm:IntroMainThmSG}
There exists a connected $n$-dimensional submanifold (with boundary) $\cS_n\subseteq\bC^n_{(a_1,\ldots,a_n)}$ and an embedding
$$
\cS_n\times\bC_{(c)}\hookrightarrow\Stab H^0\Tw(\cW((\bC^\times)^n,\fs_n)); \quad (a_1,\ldots,a_n,c)\mapsto\sigma_{\mathbf{a},c}=(Z_{\mathbf{a},c},\cP_{\mathbf{a},c})
$$
such that:
\begin{enumerate}[label=(\roman*)]
    \item The central charge $Z_{\mathbf{a},c}$ is given by the integral
$$
Z_{\mathbf{a},c}(\bL)=\int_{[\bL]}\exp\left(z_1+\cdots+z_n+c+\frac{e^{a_1}}{z_1}+\cdots+\frac{e^{a_n}}{z_n}\right)\frac{\dd z_1}{z_1}\cdots\frac{\dd z_n}{z_n}.
$$
    \item If an object $\bL$ is $\sigma_{\mathbf{a},c}$-stable, then it admits a special Lagrangian representative (up to isotopy) in $((\bC^\times)^n,\omega_n)$ with respect to the holomorphic $n$-form
$$
\Omega_{\mathbf{a},c}\coloneqq\exp\left(z_1+\cdots+z_n+c+\frac{e^{a_1}}{z_1}+\cdots+\frac{e^{a_n}}{z_n}\right)\frac{\dd z_1}{z_1}\cdots\frac{\dd z_n}{z_n}.
$$
\end{enumerate}
\end{thm}

\begin{rmk}
Recent work by Stoppa \cite{Stoppa2505, Stoppa2508} also investigates the link between Bridgeland stability and the special Lagrangian condition in related contexts. Stoppa’s approach utilizes the theory of deformed Hermitian--Yang--Mills equations \cite{CollinsYau} and primarily addresses stability conditions near the large complex structure limit.

In contrast, the stability conditions constructed in Theorem~\ref{thm:IntroMainThmSG} are \emph{away} from the large complex structure limit: Under the homological mirror symmetry equivalence (see Lemma~\ref{lemma:Mirror-(P1^)n-H0TwW})
$$
\mathrm{H}^0\Tw(\cW((\bC^\times)^n, \fs_n)) \cong \Db((\bP^1)^n),
$$
the skyscraper sheaves of points on $(\bP^1)^n$ are not stable with respect to these stability conditions (see Remark~\ref{rmk:Sn-not-geometric}). 
Theorem~\ref{thm:IntroMainThmSG} may therefore be viewed as a first step toward a higher-dimensional ``stable implies special Lagrangian" for \emph{non-geometric} stability conditions.
\end{rmk}

\begin{rmk}\label{rmk:motivation-product}
We now outline the heuristic motivation for Theorem~\ref{thm:IntroMainThmSG}, noting that the following discussion is largely conjectural.
Suppose we are given $n$ toric Fano manifolds $X_1, \ldots, X_n$ of dimensions $d_1, \ldots, d_n$, respectively. 
Let their mirror superpotentials be
$$
W^{(i)}_{\mathbf{a}^{(i)}} = \sum_{v \in V^{(i)}} a_v^{(i)} (z^{(i)})^v,
$$
where $V^{(i)} \subseteq \bZ^{d_i}$ is the set of primitive generators of the 1-cones of the fan defining $X_i$ (augmented by the origin), and $\mathbf{a}^{(i)} = \{a_v^{(i)}\}_{v \in V^{(i)}} \subseteq \bC$ is a set of coefficients. The folklore conjecture predicts that the holomorphic $d_i$-form
$$
\Omega_{\mathbf{a}^{(i)}}\coloneqq\exp\left(W^{(i)}_{\mathbf{a}^{(i)}}\right)\frac{\dd z_1^{(i)}}{z_1^{(i)}}\cdots\frac{\dd z_{d_i}^{(i)}}{z_{d_i}^{(i)}}
$$
induces a stability condition on the Fukaya--Seidel-type category mirror to $\Db(X_i)$, and stable objects are those admitting special Lagrangian representatives with respect to $\Omega_{\mathbf{a}^{(i)}}$. 
Let $\Stab^\LG \Db(X_i) \subseteq \Stab \Db(X_i)$ be the subset of stability conditions arising in this way.

The mirror of the product $X_1 \times \cdots \times X_n$ is expected to be the Landau--Ginzburg model
$$
\left( (\bC^\times)^{\sum d_i}, \sum_{i=1}^n W^{(i)}_{\mathbf{a}^{(i)}} \right).
$$
The folklore conjecture suggests that the product form
$$
\Omega_{\text{prod}} \coloneqq \exp\left( \sum_{i=1}^n W^{(i)}_{\mathbf{a}^{(i)}} \right) \bigwedge_{i=1}^n \left( \frac{\dd z_1^{(i)}}{z_1^{(i)}} \wedge \cdots \wedge \frac{\dd z_{d_i}^{(i)}}{z_{d_i}^{(i)}} \right)
$$
determines a stability condition on the corresponding product category.
Because the holomorphic form $\Omega_{\text{prod}}$ splits, such a stability condition should satisfy two key properties:
\begin{enumerate}[label=(\alph*)]
\item The central charge is multiplicative: $Z([L_1 \times \cdots \times L_n]) = \prod_{i=1}^n Z_i([L_i])$.
\item If each $L_i$ admits a special Lagrangian representative with respect to $\Omega_{\mathbf{a}^{(i)}}$, then the product $L_1 \times \cdots \times L_n$ admits a special Lagrangian representative with respect to $\Omega_{\text{prod}}$.
\end{enumerate}
This heuristic motivates the definition of \emph{product-type} stability conditions (Definition~\ref{defn:product-type-intro}). Furthermore, it predicts an embedding
$$
\prod_{i=1}^n \left( \Stab^\LG \Db(X_i) / \bC \right) \hookrightarrow \Stab^\LG \Db(X_1 \times \cdots \times X_n) / \bC,
$$
where we quotient by the free $\bC$-actions corresponding to the constant terms $a_0^{(i)}$ in the superpotentials.

Note that the same heuristic applies to the Calabi--Yau setting for stability conditions arising from the holomorphic volume form, as in Conjecture~\ref{conj:Joyce}. One may therefore expect a similar embedding to exist for products of Calabi--Yau manifolds.
\end{rmk}

The preceding discussion motivates the following definition.

\begin{defn}\label{defn:product-type-intro}
A stability condition $\sigma\in\Stab\Db(X_1\times\cdots\times X_n)$ is said to be of \emph{product-type} if there exist stability conditions $\sigma_i=(Z_i,\cP_i)\in\Stab\Db(X_i)$ for each $i=1,\ldots,n$ such that:
\begin{enumerate}[label=(\roman*)]
    \item The central charge of an external tensor product satisfies 
    $$
    Z(E_1\boxtimes\cdots\boxtimes E_n)=\prod_{i=1}^n Z_i(E_i) \quad \text{ for all objects }  E_i\in\Db(E_i).
    $$
    \item If each $E_i\in\Db(X_i)$ is $\sigma_i$-stable of phase $\phi_i$, then $E_1\boxtimes\cdots\boxtimes E_n\in\Db(X_1\times\cdots\times X_n)$ is $\sigma$-stable of phase $\sum_{i=1}^n\phi_i$.
\end{enumerate}
\end{defn}

The proof of Theorem~\ref{thm:IntroMainThmSG} is provided in Section~\ref{sec:Proof-of-main-thm}.
The logic of our argument is in the reverse direction to the motivation outlined in Remark~\ref{rmk:motivation-product}.
We first construct a family of product-type stability conditions on $\Db((\bP^1)^n)$ with the following key property: an object $E \in \Db((\bP^1)^n)$ is $(\sigma_1 \boxtimes \cdots \boxtimes \sigma_n)$-stable if and only if it decomposes as $E \cong E_1 \boxtimes \cdots \boxtimes E_n$, where each factor $E_i \in \Db(\bP^1)$ is $\sigma_i$-stable.
Transporting these stability conditions through the homological mirror symmetry equivalence, we obtain a family of stability conditions on the Fukaya--Seidel-type category whose stable objects are isotopic to products of the mirrors of the $E_i$'s. By Theorem~\ref{thm:intro-HKK}, each mirror factor admits a special Lagrangian representative with respect to the holomorphic $1$-form $\exp(z_i + \frac{e^{a_i}}{z_i}) \frac{\dd z_i}{z_i}$. It follows that the product of these Lagrangians is special with respect to the total $n$-form $\Omega_{\mathbf{a},c}$, thereby establishing the ``stable implies special Lagrangian" correspondence.

We verify the heuristic proposed in Remark~\ref{rmk:motivation-product}, regarding the embedding of stability spaces for products of Fano and Calabi--Yau manifolds via product-type stability conditions, for the following two-dimensional cases.

\begin{thm}[see Theorems~\ref{thm:P1xP1} and \ref{thm:E1xE2}]
\label{thm:embedding-P1-E}
There exists an embedding
$$
\left(\Stab\Db(\bP^1)/\bC\right)\times\left(\Stab\Db(\bP^1)/\bC\right)\hookrightarrow\Stab\Db(\bP^1\times\bP^1)/\bC
$$
which sends a pair of stability conditions $(\overline{\sigma_1}, \overline{\sigma_2})$ to the unique product-type stability condition $\overline{\sigma_1 \boxtimes \sigma_2}$.
Analogously, for any two elliptic curves $X_1$ and $X_2$, there exists an embedding
$$
\left(\Stab\Db(X_1)/\bC\right)\times\left(\Stab\Db(X_2)/\bC\right)\hookrightarrow\Stab\Db(X_1\times X_2)/\bC
$$
via product-type stability conditions.
\end{thm}

\begin{rmk}
We conclude the Introduction with some further questions worth investigating:
\begin{enumerate}[label=(\roman*)]
    \item The converse of Theorem~\ref{thm:IntroMainThmSG}, namely, whether every special Lagrangian submanifold of $((\bC^\times)^n,\omega_n)$ with respect to the holomorphic $n$-form
$$
\Omega_{\mathbf{a},c}\coloneqq\exp\left(z_1+\cdots+z_n+c+\frac{e^{a_1}}{z_1}+\cdots+\frac{e^{a_n}}{z_n}\right)\frac{\dd z_1}{z_1}\cdots\frac{\dd z_n}{z_n}.
$$
    is $\sigma_{\mathbf{a},c}$-stable, is not known. A similar difficulty (for different reasons) also appears in a recent work of Stoppa \cite{Stoppa2508}, where under suitable conditions, Bridgeland stability implies the existence of a \emph{deformed Hermitian--Yang--Mills connection} (mirror to special Lagrangian condition near the large volume limit), but the converse is not known.

    In our setting, the simplest testing examples would be to take any $a\in\bC$ giving an algebraic stability condition on $H^0\Tw(\cW(\bC^\times,\fs))\cong\Db(\bP^1)$, and investigate special Lagrangian submanifolds of $((\bC^\times)^2,\omega_2)$ with respect to the holomorphic $2$-form
$$
\exp\left(z_1+z_2+\frac{e^{a}}{z_1}+\frac{e^{a}}{z_2}\right)\frac{\dd z_1}{z_1}\frac{\dd z_2}{z_2}.
$$
    \item In Remark~\ref{rmk:Diagonal_stable}, we note that for product-type stability conditions on $\bP^1\times\bP^1$ arising from the product of two \emph{geometric} stability conditions on $\bP^1$, there are stable objects which are not isomorphic to $E_1\boxtimes E_2$ for some $E_1,E_2\in\Db(\bP^1)$, unlike those constructed in Theorem~\ref{thm:IntroMainThmSG}.
    It would be interesting to see whether the mirror of such an object (which is not a product of two Lagrangians) is special Lagrangian with respect to, say,
$$
\exp\left(z_1+z_2+\frac{e^{b}}{z_1}+\frac{e^{b}}{z_2}\right)\frac{\dd z_1}{z_1}\frac{\dd z_2}{z_2},
$$
    where $b\in\bC$ gives a geometric stability condition on $H^0\Tw(\cW(\bC^\times,\fs))\cong\Db(\bP^1)$.
    \item Heuristics in Remark~\ref{rmk:motivation-product} suggest an embedding
$$
\bC^n\cong(\Stab\Db(\bP^1)/\bC)^n\hookrightarrow\Stab\Db((\bP^1)^n)/\bC
$$
    whose image consists of product-type stability conditions as in Definition~\ref{defn:product-type-intro}. The case $n=2$ is proved in Theorem~\ref{thm:embedding-P1-E}. For $n\geq3$, one would need to show the stability of line bundles $\cO(k_1,\ldots,k_n)$ with respect to the geometric stability conditions.
    \item More generally, it is natural to ask whether there always exists an embedding
$$
(\Stab\Db(X_1)/\bC)\times(\Stab\Db(X_2)/\bC)\hookrightarrow\Stab\Db(X_1\times X_2)/\bC
$$
    given by product-type stability conditions satisfying the conditions in Definition~\ref{defn:product-type-intro}, although the symplectic geometric heuristic in Remark~\ref{rmk:motivation-product} would no longer be valid.
\end{enumerate}
\end{rmk}

\medskip\noindent\textbf{Acknowledgment.}
The author would like to thank Fabian~Haiden, Hansol~Hong, Yu-Shen~Lin, Yat-Hin~Marco~Suen, and Peng~Zhou for helpful conversations and correspondences. 
The author would especially like to thank Jacopo~Stoppa for explaining his work and for valuable discussions.
The author was partially supported by the National Natural Science Foundation of China (Grant No.~12401053).


\section{Preliminaries.}
In Section~\ref{subsec:HMS}, we review the homological mirror symmetry equivalence between $\Db(\bP^1)$ and its mirror, and use the categorical K\"unneth formulae to extend it to the product $\Db((\bP^1)^n)$.
In Section~\ref{subsec:BSC}, we recall the definition of Bridgeland stability conditions, review facts on stability conditions on $\Db(\bP^1)$, and then discuss the constructions of stability conditions from exceptional collections and semiorthogonal decompositions.

\subsection{HMS for \texorpdfstring{$(\bP^1)^n$}{} and its mirror.}
\label{subsec:HMS}
We model the Fukaya--Seidel-type categories using the framework of partially wrapped Fukaya categories with stops. Given a Liouville manifold $X$ (an exact symplectic manifold with a cylindrical structure at infinity) and a closed subset $\fs \subseteq \partial_\infty X$ (the stop), the partially wrapped Fukaya category $\cW(X, \fs)$ is the $A_\infty$-category defined in the work of Ganatra, Pardon, and Shende \cite{GPS20, GPS}. Its objects are Lagrangians $L\subseteq X$ that are eventually conical and disjoint from $\mathfrak{s}$ at infinity; its morphism complexes are Floer cochains after wrapping Lagrangians in the complement of $\mathfrak{s}$. 
We refer the reader to \cite{GPS20, GPS} for further details and properties.

One of the simplest examples occurs when $X = \bC^\times \cong T^* S^1$. Taking $\fs$ to be the co-0-sphere over a fixed point $x \in S^1$, there is an equivalence \cite[Example 1.22]{GPS}:
$$
\cW(\bC^\times,\fs)\cong\Perf(\bP^1).
$$
By the K\"unneth formula for partially wrapped Fukaya categories \cite[Corollary 1.18]{GPS}, there is an equivalence
$$
\cW(\bC^\times, \fs)^{\otimes n} \cong \cW((\bC^\times)^n, \fs_n),
$$
where $\fs_n \subseteq \partial_\infty (\bC^\times)^n$ is the product stop (cf.~\cite[Section~1.2]{GPS}).
On the level of objects, this equivalence maps the tensor product $L_1 \otimes \cdots \otimes L_n$ to a cylindrization $\widetilde{L_1 \times \cdots \times L_n}$, which is a Lagrangian isotopic to the product $L_1 \times \cdots \times L_n$ (cf.~\cite[Section~8.3]{GPS}).

On the other hand, there is an equivalence of perfect complexes (cf.~\cite[Theorem~1.2]{BZFN}):
$$
\Perf(\bP^1)^{\otimes n} \cong \Perf((\bP^1)^n),
$$
determined by the external tensor product $E_1 \otimes \cdots \otimes E_n \mapsto E_1 \boxtimes \cdots \boxtimes E_n$.
Combining these results yields the following lemma.

\begin{lemma}
\label{lemma:Mirror-(P1^)n-H0TwW}
There exists an equivalence of triangulated categories
$$
\Mir_n \colon \Db((\bP^1)^n) \xrightarrow{\sim} H^0 \mathrm{Tw} \cW((\bC^\times)^n, \fs_n),
$$
such that the image $\Mir_n(\cO(k_1, \ldots, k_n))$ admits a representative which is isotopic to the product Lagrangian
$$
L_{k_1} \times \cdots \times L_{k_n} \subseteq (\bC^\times)^n,
$$
where each $L_{k_i}\subseteq\bC^\times$ denotes the Lagrangian mirror to the line bundle $\cO(k_i)$ on $\bP^1$.
\end{lemma}

\subsection{Bridgeland stability conditions.}
\label{subsec:BSC}

\subsubsection{Definition of Bridgeland stability conditions.}
We recall the definition of Bridgeland stability conditions and some of their basic properties.

Let $\cD$ be a triangulated category. Fix a free abelian group $\Lambda$ of finite rank, a group homomorphism $v\colon K_0(\cD)\rightarrow\Lambda$, and a norm $||\cdot||$ on $\Lambda\otimes_\bZ\bR$.
\begin{defn}[\cite{BriStab}]
A \emph{stability condition} $\sigma=(Z,\cP)$ on $\cD$ (with respect to $\Lambda$ and $v$) consists of:
\begin{itemize}
    \item a group homomorphism $Z\colon\Lambda\rightarrow\bC$, called the \emph{central charge}, and
    \item a collection $\cP=\{\cP(\phi)\}_{\phi\in\bR}$ of full additive subcategories of $\cD$, where nonzero objects in $\cP(\phi)$ are called \emph{semistable objects of phase $\phi$},
\end{itemize}
satisfying the following axioms:
\begin{enumerate}[label=(\roman*)]
    \item If $0\neq E\in \cP(\phi)$, then $Z(v(E))\in\bR_{>0}\cdot e^{i\pi\phi}$.
    \item For all $\phi\in\bR$, $\cP(\phi+1)=\cP(\phi)[1]$.
    \item If $\phi_1>\phi_2$ and $E_i\in \cP(\phi_i)$, then $\Hom(E_1,E_2)=0$.
    \item Every nonzero object $E\in\cD$ admits a Harder--Narasimhan filtration: a finite sequence of exact triangles
\begin{equation*}
\xymatrix{
0=E_0 \ar[r] & E_1\ar[d]  \ar[r] & E_2 \ar[r] \ar[d]& \cdots \ar[r]& E_{k-1}\ar[r] & E \ar[d]\\
                    & A_1 \ar@{-->}[lu]& A_2 \ar@{-->}[lu] &  & & A_k  \ar@{-->}[lu] 
}    
\end{equation*}
    where $A_i\in \cP(\phi)$ and $\phi_1>\cdots>\phi_k$.
    \item (Support property) There exists a constant $C>0$ such that 
    $$
    |Z(v(E))|\geq C\cdot||v(E)||
    $$
    holds for every semistable object $E\in \cup_\phi \cP(\phi)$.
\end{enumerate}
The set of such stability conditions is denoted by $\Stab_\Lambda(\cD)$.
\end{defn}

\begin{rmk}
Given a stability condition $\sigma=(Z,\cP)$, let
$$
\cA_\sigma \coloneqq\cP(0,1]
$$
denote the associated \emph{heart} (of a bounded t‑structure) on $\cD$.
It is the extension‑closed subcategory of $\cD$ generated by the semistable objects $\cP(\phi)$ with phases $\phi \in (0,1]$.

A stability condition on $\cD$ can be specified by a pair $(Z,\cA)$ consisting of a heart $\cA \subseteq \cD$ and a stability function $Z$ satisfying the Harder--Narasimhan property and the support property \cite[Proposition~5.3]{BriStab}. Hence we may denote the same stability condition by either $(Z,\cP)$ or $(Z,\cA)$.
\end{rmk}

The stability space carries a natural (generalized) metric \cite[Proposition~8.1]{BriStab}, inducing a natural topology. The main structural result of \cite{BriStab} is that $\Stab_\Lambda(\cD)$ forms a complex manifold \cite[Theorem~7.1]{BriStab}. Explicitly, the forgetful map
$$
\Stab_\Lambda(\cD)\rightarrow\Hom(\Lambda,\bC); \quad \sigma=(Z,\cP)\mapsto Z
$$
is a local homeomorphism when $\Hom(\Lambda,\bC)$ is given the linear topology. Therefore, $\Stab_\Lambda(\cD)$ is naturally a complex manifold of dimension $\rank(\Lambda)$.

The space $\Stab_\Lambda(\cD)$ carries natural group actions by autoequivalences of $\cD$ and by the universal cover $\widetilde{\GL^+(2,\bR)}$.
In the following, we will only use the action of the subgroup $\mathbb{C} \subseteq \widetilde{\GL^+(2,\mathbb{R})}$, defined as follows:
for $c\in\bC$,
$$
\sigma=(Z,\cP)\mapsto\sigma\cdot c\coloneqq\left(e^c\cdot Z, \cP'\right), \quad \text{ where } \cP'(\phi)=\cP\left(\phi-\frac{1}{\pi}\text{Im}(c)\right).
$$
The real part of $c$ scales the central charges, while the imaginary part shifts the phases.
This $\bC$-action is free on $\Stab_\Lambda(\cD)$.

\begin{rmk}
When $\cD =\Db(X)$ is the bounded derived category of coherent sheaves on a smooth projective variety $X$, the standard choice for $\Lambda$ is the numerical Grothendieck group $N(\Db(X))$. For $X$ a smooth projective curve, we have $N(\Db(X))=\bZ\oplus\bZ$, and the natural projection $K_0(\Db(X))\to N(\Db(X))$ sends each class to its rank and degree.
\end{rmk}

\subsubsection{Bridgeland stability conditions on \texorpdfstring{$\bP^1$}{} and its mirror.}
We recall the classification of stability conditions on $\Db(\mathbb{P}^1)$ due to Okada \cite{Okada}. Stability conditions on $\Db(\mathbb{P}^1)$ fall into two distinct classes:
\begin{enumerate}[label=(\alph*)]
    \item \emph{Geometric stability conditions:} those for which all line bundles and skyscraper sheaves are stable. The set of such stability conditions is denoted
    $$
    \Stab^\Geo\Db(\bP^1)\subseteq\Stab\Db(\bP^1).
    $$
    \item \emph{Algebraic stability conditions:} those for which, for some $k \in \mathbb{Z}$, only $\cO(k)$, $\cO(k+1)$, and their shifts are stable. For each $k \in \mathbb{Z}$ we write
    $$
    \Stab^\Alg_k\Db(\bP^1)\subseteq\Stab\Db(\bP^1).
    $$
    for the corresponding subset.
\end{enumerate}
Okada \cite{Okada} shows that there is a decomposition
$$
\Stab\Db(\bP^1)=\Stab^\Geo\Db(\bP^1)\coprod\left(\coprod_{k\in\bZ}\Stab^\Alg_k\Db(\bP^1)\right),
$$
and establishes the following isomorphisms:
$$
\Stab\Db(\bP^1)\cong\bC^2, \quad \Stab^\Geo\Db(\bP^1)\cong\bH\times\bC, \quad \Stab^\Alg_k\Db(\bP^1)\cong\overline\bH\times\bC.
$$
Quotienting by the free $\mathbb{C}$-actions yield:
$$
\Stab\Db(\bP^1)/\bC\cong\bC, \quad \Stab^\Geo\Db(\bP^1)/\bC\cong\bH, \quad \Stab^\Alg_k\Db(\bP^1)/\bC\cong\overline\bH.
$$

\begin{rmk}
A key distinction between the two types of stability conditions, which will be useful in what follows, concerns the phases of stable objects:
\begin{enumerate}[label=(\alph*)]
    \item If $\sigma\in\Stab^\Geo\Db(\bP^1)$, then
    $$
    \phi_\sigma(k(x))-1<\cdots < \phi_\sigma(\cO(-1))<\phi_\sigma(\cO)<\phi_\sigma(\cO(1))<\cdots<\phi_\sigma(k(x)).
    $$
    \item If $\sigma\in\Stab^\Alg_k\Db(\bP^1)$, then
    $$
    \phi_\sigma(\cO(k))+1\leq\phi_\sigma(\cO(k+1)).
    $$
\end{enumerate}
\end{rmk}

The work of Haiden–Katzarkov–Kontsevich \cite{HKK}, when specialized to the setting of Theorem~\ref{thm:intro-HKK}, provides a geometric interpretation of Okada's classification. Fix a homological mirror symmetry equivalence as in Section~\ref{subsec:HMS}:
$$
\Mir\colon\Db(\bP^1)\xrightarrow{\sim}H^0\Tw\cW(\bC^\times,\fs).
$$
Then \cite[Theorems 5.2 and 6.3]{HKK} (see also Theorem~\ref{thm:intro-HKK}) yields the following corollary.

\begin{cor}\label{cor:StabDbP^1viaHKK}
There exists an isomorphism
$$
\bC^2_{(a,c)}\xrightarrow{\sim}\Stab\Db(\bP^1); \quad
(a,c)\mapsto\sigma_{a,c}=(Z_{a,c},\cP_{a,c})
$$
with the following properties:
\begin{enumerate}[label=(\roman*)]
    \item The central charge $Z_{a,c}$ is given by the integral
$$
Z_{a,c}(E)=\int_{[\Mir(E)]}\exp\left(z+c+\frac{e^a}{z}\right)\frac{\dd z}{z}.
$$
    \item If $E$ is $\sigma_{a,c}$-stable, then $\Mir(E)$ admits a special Lagrangian representative with respect to the holomorphic $1$-form
$$
\Omega_{a,c}\coloneqq\exp\left(z+c+\frac{e^a}{z}\right)\frac{\dd z}{z}.
$$
\end{enumerate}
\end{cor}

\begin{rmk}
\label{rmk:slice-C-Stab(P1)}
Some comments on Corollary~\ref{cor:StabDbP^1viaHKK}:
\begin{enumerate}[label=(\alph*)]
    \item Regarding the parameters $(a, c) \in \bC^2$: The parameter $c$ corresponds to the free $\mathbb{C}$-action on $\Stab\Db(\mathbb{P}^1)$; varying $c$ does not change the set of stable objects. 

    The parameter $a$ carries the essential stability data. The monodromy of the term $e^{a}$ reflects the autoequivalence $(-\otimes\cO(1))$ on $\Db(\mathbb{P}^1)$.

    \item\label{item:slice-C-Stab(P1)} Setting $c=0$ gives a slice of the free $\mathbb{C}$-action, which we denote
    $$
    \Stab^{c=0}\Db(\bP^1)\subseteq\Stab\Db(\bP^1).
    $$ 
    The composition
    $$
    \Stab^{c=0}\Db(\bP^1) \hookrightarrow \Stab\Db(\bP^1) \to \Stab\Db(\bP^1)/\bC
    $$
    is an isomorphism.
\end{enumerate}
\end{rmk}

\subsubsection{Stability conditions from exceptional collections.}
\label{subsubsection:StabFromExcep-Macri}
We recall the construction of stability conditions from exceptional collections, following \cite[Section~3.3]{Macri}.

An object $E$ in a triangulated category $\cD$ is called \emph{exceptional} if
$$
\Hom(E,E[k])=
\begin{cases}
    0 & \text{ if } k\neq0, \\
    \bC & \text{ if } k=0.
\end{cases}
$$
An ordered set $\cE = \{E_1, \ldots, E_n\}$ is an \emph{exceptional collection} if each $E_i$ is exceptional and
$$
\Hom(E_i,E_j[k])=0 \quad \text{ for all } i>j \text{ and all } k\in\bZ.
$$

\begin{defn}[\cite{Macri}]
Let $\cE=\{E_1, \ldots, E_n\}$ be an exceptional collection in $\cD$.
\begin{enumerate}[label=(\alph*)]
    \item $\cE$ is \emph{strong} if $\Hom(E_i,E_j[k])=0$ for all $i,j$ and all $k\neq0$.
    \item $\cE$ is \emph{full} (or \emph{complete}) if it generates $\cD$ under extensions and shifts.
    \item $\cE$ is \emph{$\text{Ext}$-exceptional} if $\Hom(E_i,E_j[k])=0$ for all $i\neq j$ and all $k\leq0$.
\end{enumerate}
\end{defn}

\begin{lemma}[{\cite[Section~3.3]{Macri}}]
\label{lemma:Macri}
Let $\cE=\{E_1,\ldots,E_n\}$ be a full Ext-exceptional collection in $\cD$. Then:
\begin{enumerate}[label=(\alph*)]
    \item The extension-closed subcategory $\left<E_1,\ldots,E_n\right>$ is the heart of a bounded t-structure on $\cD$.
    \item For any choice of $z_1,\ldots,z_n\in\bH\cup\bR_{<0}$, there exists a unique stability condition $\sigma$ on $\cD$ such that each $E_i$ is $\sigma$-stable, $Z(E_i)=z_i$, and $0<\phi_\sigma(E_i)\leq1$.
\end{enumerate}
\end{lemma}

\subsubsection{Stability conditions from semiorthogonal decompositions.}
\label{subsubsec:stab-from-SOD}
We recall the gluing construction of stability conditions along a semiorthogonal decomposition, following \cite{CPgluing}. This generalizes the method of constructing stability conditions from exceptional collections described in Section~\ref{subsubsection:StabFromExcep-Macri}.

Let $\cD=\left<\cD_1,\cD_2\right>$ be a \emph{semiorthogonal decomposition}, i.e.~$\cD_1$ and $\cD_2$ are triangulated subcategories of $\cD$ satisfying $\Hom(E_2, E_1) = 0$ for all $E_1 \in \cD_1$, $E_2 \in \cD_2$, and every object $E \in \cD$ fits into an exact triangle
$$
E_2\to E\to E_1\to E_2[1]
$$
with $E_1\in\cD_1$ and $E_2\in\cD_2$.
The objects $E_1$ and $E_2$ in this triangle are functorial in $E$: $E_1 = \lambda_1(E)$, where $\lambda_1$ is the left adjoint to the inclusion $\cD_1 \hookrightarrow \cD$, and $E_2 = \rho_2(E)$, where $\rho_2$ is the right adjoint to the inclusion $\cD_2 \hookrightarrow \cD$.

We first recall how hearts glue across a semiorthogonal decomposition.

\begin{lemma}[{\cite[Lemma~2.1]{CPgluing}}]
\label{lemma:CP-gluing-heart}
Let $\cD=\left<\cD_1,\cD_2\right>$ be a semiorthogonal decomposition, and let $\cA_i \subseteq \cD_i$ be hearts such that
$$
\Hom\left(\cA_1,\cA_2[\leq0]\right)=0.
$$
Then the subcategory
$$
\cA=\{E\in\cD:\lambda_1(E)\in\cA_1,\rho_2(E)\in\cA_2\}
$$
is the heart of a bounded $t$-structure on $\cD$.
Moreover, every object $E \in \cA$ sits in a short exact sequence in $\cA$
$$
0\to\rho_2(E)\to E\to\lambda_1(E)\to0,
$$
where $\lambda_1(E)\in\cA_1$ and $\rho_2(E)\in\cA_2$.
\end{lemma}

The next proposition provides a criterion for the glued heart and central charge to satisfy the Harder--Narasimhan property.

\begin{prop}[{\cite[Proposition~3.5(a)]{CPgluing}}]
\label{prop:CP-gluing-stab-HN}
Let $\cD = \langle \cD_1, \cD_2 \rangle$ be a semiorthogonal decomposition, and let $\sigma_i = (Z_i, \cA_i)$ be stability conditions on $\cD_i$ ($i=1,2$).
Assume that 
$$
\Hom\left(\cA_1,\cA_2[\leq0]\right)=0.
$$
Let $\cA \subseteq \cD$ be the heart obtained by gluing $\cA_1$ and $\cA_2$ as in Lemma~\ref{lemma:CP-gluing-heart}, and define a stability function on $\cA$ by
$$
Z(E)\coloneqq Z_1(\lambda_1(E))+Z_2(\rho_2(E)).
$$
Assume in addition that $0$ is an isolated point of $\text{Im}Z_i(\cA_i)\subseteq\bR_{\geq0}$ for $i=1,2$.
Then $(Z,\cA)$ satisfies the Harder--Narasimhan property.
\end{prop}

\section{Proof of Theorem~\ref{thm:IntroMainThmSG}.}
\label{sec:Proof-of-main-thm}
In this section, we prove Theorem~\ref{thm:main-AG-Sn}. As outlined in the introduction, we first construct product-type stability conditions on $\Db((\bP^1)^n)$ with the property that every stable object splits as $E_1\boxtimes\cdots\boxtimes E_n$, where each $E_i$ is stable in the corresponding factor. This construction is carried out in Section~\ref{subsection:Pf-of-Main-AG}.
In Section~\ref{subsection:Pf-of-Main-SG}, we transport these stability conditions through the homological mirror symmetry equivalence to the Fukaya--Seidel‑type categories, and complete the proof of Theorem~\ref{thm:IntroMainThmSG}.

\subsection{Product-type stability conditions on \texorpdfstring{$(\bP^1)^n$}{}.}
\label{subsection:Pf-of-Main-AG}
Consider the subset $\cS_n$ of $n$-tuples of stability conditions on $\Db(\bP^1)$ defined by
$$
\cS_n=\{(\sigma_1,\ldots,\sigma_n):\text{ at most one }\sigma_i\text{ lies in }\Stab^\Geo\Db(\bP^1)\}\subseteq(\Stab\Db(\bP^1))^n.
$$
Quotienting by the component-wise free $\bC$-actions yields the space
$$
\overline{\cS}_n\coloneqq\cS_n/\bC^n\subseteq(\Stab\Db(\bP^1)/\bC)^n.
$$
It is straightforward to verify that $\overline\cS_n$ is a connected $n$-dimensional submanifold (with boundary) of $(\Stab\Db(\bP^1)/\bC)^n\cong\bC^n$.

The goal of this subsection is to prove the following theorem.

\begin{thm}\label{thm:main-AG-Sn}
For every $(\sigma_1,\ldots,\sigma_n)\in\cS_n$, there exists a unique stability condition, denoted by $\sigma_1\boxtimes\cdots\boxtimes\sigma_n\in\Stab\Db((\bP^1)^n)$, such that:
\begin{enumerate}[label=(\roman*)]
\item\label{cond:prod-Z} Its central charge satisfies
$$
Z_{\sigma_1\boxtimes\cdots\boxtimes\sigma_n}(E_1\boxtimes\cdots\boxtimes E_n)=\prod_{i=1}^n Z_{\sigma_i}(E_i)
$$
for all objects $E_1,\ldots,E_n\in\Db(\bP^1)$.
\item\label{cond:prod-phase} If each $E_i\in\Db(\bP^1)$ is $\sigma_i$-stable of phase $\phi_i$, then $E_1\boxtimes\cdots\boxtimes E_n$ is $(\sigma_1\boxtimes\cdots\boxtimes\sigma_n)$-stable of phase $\sum_{i=1}^n\phi_i$.
\end{enumerate}
Moreover, these product-type stability conditions satisfy:
\begin{enumerate}[label=(\alph*)]
\item\label{item-property-a} An object $E\in\Db((\bP^1)^n)$ is $(\sigma_1\boxtimes\cdots\boxtimes\sigma_n)$-stable if and only if $E\cong E_1\boxtimes\cdots\boxtimes E_n$ for some objects $E_i\in\Db(\bP^1)$ that are $\sigma_i$-stable for all $i=1,\ldots,n$.
\item\label{item-property-b} Two elements $(\sigma_1,\ldots,\sigma_n),(\sigma'_1,\ldots,\sigma'_n)\in\cS_n$ yield the same product stability condition up to the $\bC$-action, i.e.,
$$
\overline{\sigma_1\boxtimes\cdots\boxtimes\sigma_n}=\overline{\sigma'_1\boxtimes\cdots\boxtimes\sigma'_n}\in\Stab\Db((\bP^1)^n)/\bC,
$$
if and only if $(\overline{\sigma_1},\ldots,\overline{\sigma_n})=(\overline{\sigma'_1},\ldots,\overline{\sigma'_n})$ in $\overline{\cS}_n$.
\end{enumerate}
Consequently, there is a well-defined embedding
$$
\overline{\cS}_n\hookrightarrow\Stab\Db((\bP^1)^n)/\bC.
$$
\end{thm}

We first show that for each $(\sigma_1,\ldots,\sigma_n) \in \cS_n$, there exists a unique stability condition on $\Db((\bP^1)^n)$ satisfying Conditions~\ref{cond:prod-Z} and~\ref{cond:prod-phase} of Theorem~\ref{thm:main-AG-Sn}, and that this stability condition further satisfies Property~\ref{item-property-a}.

\medskip\noindent\textbf{Case 1: All $\sigma_i\notin\Stab^\Geo\Db(\bP^1)$.}
In this case, there exist integers $k_1,\ldots,k_n$ such that 
$$
\sigma_i\in\Stab^\Alg_{k_i}\Db(\bP^1) \quad \text{ for all } \quad 1\leq i\leq n.
$$
To simplify the notation, we introduce the following definitions:
\begin{itemize}
\item For an $n$-tuple $I=(i_1,\ldots,i_n)\in\{0,1\}^n$, let
$$
\cL_I=\cL_{i_1,\ldots,i_n}\coloneqq\cO(k_1+i_1,\ldots,k_n+i_n).
$$
\item Let $|I|\coloneqq\sum_{\ell=1}^n i_\ell$.
\item We define a partial ordering on $\{0,1\}^n$ by setting $I\leq I'$ if $i_\ell\leq i'_\ell$ holds for all $1\leq \ell\leq n$.
\end{itemize}

The ordered set of $2^n$ distinct line bundles $\{\cL_{I_1}, \ldots, \cL_{I_{2^n}}\}$ forms an exceptional collection in $\Db((\bP^1)^n)$ if and only if the ordering is compatible with the partial order on indices, in the sense that $I_r \leq I_s$ implies $r \leq s$. 
This follows from the observation that $\Hom(\cL_I, \cL_{I'}[k])\neq0$ precisely when $I \leq I'$ and $k=0$. 
Such an ordering always exists (for instance, by arranging the line bundles by non‑decreasing total degree $|I|$); we fix such an ordering for the rest of the argument.

\begin{prop}\label{prop:existence-case1}
For any positive numbers $m_I>0$, any real number $\phi\in\bR$, and real numbers $\phi_1,\ldots,\phi_n\geq1$, there exists a unique stability condition on $\Db((\bP^1)^n)$ such that:
\begin{enumerate}[label=(\roman*)]
    \item Each $\cL_I$ is stable of phase $\phi_I\coloneqq\phi+\sum_{\ell=1}^n i_\ell\cdot\phi_\ell$.
    \item $Z(\cL_I)=m_I\, e^{\sqrt{-1}\pi\phi_I}$.
\end{enumerate}
Moreover, with respect to this stability condition, the only stable objects are the line bundles $\cL_I$ and their shifts.
\end{prop}

\begin{proof}
Set $p_I=-\lceil\phi_I-1\rceil$, so that the shifted phases satisfy $\phi(\cL_I[p_I]) = \phi_I + p_I \in (0, 1]$ for all $I \in \{0,1\}^n$. We claim that the shifted collection
$$
\cE\coloneqq\{\cL_{I_1}[p_{I_1}],\ldots,\cL_{I_{2^n}}[p_{I_{2^n}}]\}
$$
is a full Ext-exceptional collection. It suffices to show that
$$
\Hom^{\leq0}(\cL_I[p_I],\cL_{I'}[p_{I'}])=0 \quad \text{ for all } \quad I\neq I'.
$$
A non-zero morphism exists only if $I \lneq I'$ in the partial order. In this case, $\phi_{I'} \geq \phi_I + \phi_i$ for some $1 \leq i \leq n$. Given the assumption $\phi_i \geq 1$, we have $\phi_{I'} \geq \phi_I + 1$, which implies $p_I \geq p_{I'} + 1$.
Therefore, for $k\leq0$ we have
$$
\Hom^k(\cL_I[p_I], \cL_{I'}[p_{I'}]) = \Hom^{k + p_{I'} - p_I}(\cL_I, \cL_{I'})=0
$$
since $k + p_{I'} - p_I<0$.
Thus, the existence and uniqueness of the stability condition follow from \cite[Section~3.3]{Macri} (see Lemma~\ref{lemma:Macri}), where the heart $\cA$ is generated by the collection $\cE$. Furthermore, the objects $\cL_I$ are stable with respect to this stability condition.

It remains to show that, up to shifts, the line bundles $\cL_I$ are the only stable objects. It suffices to prove that the only stable objects in the heart $\cA = \langle \cE \rangle$ are precisely the elements of $\cE$. Every object $E \in \cA$ admits a finite filtration in the heart$$0=E_0\subsetneq E_1\subsetneq\cdots\subsetneq E_m=E$$such that each subquotient $E_j/E_{j-1}$ is isomorphic to some $\cL_{I}[p_I]$.

We claim that $\Hom^1(\cL_{I'}[p_{I'}], \cL_I[p_I]) = 0$ whenever $\phi(\cL_I[p_I]) < \phi(\cL_{I'}[p_{I'}])$. Indeed, if $\Hom^1(\cL_{I'}[p_{I'}], \cL_I[p_I]) \neq 0$, then $I' \lneq I$ and $p_{I'} = p_I + 1$. Hence
$$
\phi(\cL_I[p_I]) = \phi_I + p_I\geq(\phi_{I'}+1)+(p_{I'}-1)=\phi(\cL_{I'}[p_{I'}]).
$$
This proves the claim.

Because these extensions vanish, we may rearrange the subquotients of the filtration so that their phases are non‑increasing. For instance, suppose
$$
E_i/E_{i-1}\cong\cL_I[p_I] \quad \text{ and } \quad E_{i+1}/E_i\cong\cL_{I'}[p_{I'}] \quad \text{ satisfy } \quad \phi(\cL_I[p_I]) < \phi(\cL_{I'}[p_{I'}]).
$$
By the claim, we have $\Hom^1(\cL_{I'}[p_{I'}], \cL_I[p_I]) = 0$. Therefore,
$$
E_{i+1}/E_{i-1} \cong \cL_I[p_I] \oplus \cL_{I'}[p_{I'}].
$$
We can therefore replace $E_i$ with another intermediate object $E_i'$ to obtain a new filtration where the phases of the subquotients are non‑increasing.
Thus, $E$ is stable exactly when its filtration has length one, i.e.~when $E\cong\cL_I[p_I]$ for some $I$. This completes the proof.
\end{proof}

\begin{rmk}
The argument given above is similar in spirit to \cite[Lemma~2.4]{ChunyiP2} and can be used to prove a basic fact regarding \emph{pure} algebraic stability conditions, which we record in the appendix (see Proposition~\ref{prop:appendix-pure}).
While this result is elementary, we include a complete proof for two reasons: it appears only as a specialized case in \cite[Lemma~2.4]{ChunyiP2}, and in a weaker form in recent work such as \cite[Proposition~3.7 and Remark~3.8]{Vanja}; thus we find it useful to record it explicitly.
\end{rmk}

\begin{cor}\label{cor:ALG-EUa}
Let $\sigma_1, \ldots, \sigma_n \in \Stab\Db(\bP^1)$ be stability conditions such that
$$
\sigma_i\in\Stab^\Alg_{k_i}\Db(\bP^1) \quad \text{ for all } \quad 1\leq i\leq n.
$$
Then there exists a unique stability condition $\sigma_1\boxtimes\cdots\boxtimes\sigma_n\in\Stab\Db((\bP^1)^n)$ satisfying:
\begin{enumerate}[label=(\roman*)]
    \item Its central charge satisfies $Z_{\sigma_1\boxtimes\cdots\boxtimes\sigma_n}(E_1\boxtimes\cdots\boxtimes E_n)=\prod_{i=1}^nZ_i(E_i)$ for all objects $E_1,\ldots,E_n\in\Db(\bP^1)$.
    \item If each $E_i\in\Db(\bP^1)$ is $\sigma_{i}$-stable of phase $\phi_i$, then $ E_1\boxtimes\cdots\boxtimes E_n$ is $(\sigma_1\boxtimes\cdots\boxtimes\sigma_n)$-stable of phase $\sum_{i=1}^n\phi_i$.
\end{enumerate}
Moreover, an object $E\in\Db((\bP^1)^n)$ is $(\sigma_1\boxtimes\cdots\boxtimes\sigma_n)$-stable if and only if $E\cong E_1\boxtimes\cdots\boxtimes E_n$ for some $\sigma_i$-stable objects $E_i \in \Db(\bP^1)$.
\end{cor}

\begin{proof}
For each $i \in \{1, \ldots, n\}$, the $\sigma_i$-stable objects (up to shift) are precisely $\cO(k_i)$ and $\cO(k_i+1)$.
Denote their phases and central charges as follows:
\begin{itemize}
\item $\cO(k_i)$ has phase $\psi_i$ and central charge $Z_i(\cO(k_i)) = m_{i,0} e^{\sqrt{-1}\pi\psi_i}$;
\item $\cO(k_i+1)$ has phase $\psi_i + \phi_i$ and central charge $Z_i(\cO(k_i+1)) = m_{i,1} e^{\sqrt{-1}\pi(\psi_i + \phi_i)}$.
\end{itemize}
Since $\sigma_i \in \Stab^\Alg_{k_i}\Db(\bP^1)$, we have $\phi_i \geq 1$. 
Now apply Proposition~\ref{prop:existence-case1} with the choices
$$
m_I = m_{i_1, \ldots, i_n} \coloneqq \prod_{\ell=1}^n m_{\ell, i_\ell}, \qquad \phi \coloneqq \sum_{i=1}^n \psi_i,
$$
and $\phi_1,\ldots,\phi_n$ as above.
The conclusion follows directly from the properties stated in that proposition.
\end{proof}

\medskip\noindent\textbf{Case 2: Exactly one component $\sigma_i$ lies in $\Stab^\Geo\Db(\bP^1)$.}
Without loss of generality, we may assume $\sigma_1\in\Stab^\Geo\Db(\bP^1)$ and $\sigma_2,\ldots,\sigma_n\in\Stab\Db(\bP^1)\backslash\Stab^\Geo\Db(\bP^1)$.

\begin{prop}\label{prop:GeoAlg-EUa}
Let $\sigma_1, \ldots, \sigma_n \in \Stab\Db(\bP^1)$ be stability conditions such that 
$$
\sigma_1\in\Stab^\Geo\Db(\bP^1) \quad \text{ and } \quad
\sigma_i\in\Stab^\Alg_{k_i}\Db(\bP^1) \quad \text{ for } \quad 2\leq i\leq n.
$$
Then there exists a unique stability condition $\sigma_1\boxtimes\cdots\boxtimes\sigma_n\in\Stab\Db((\bP^1)^n)$ satisfying:
\begin{enumerate}[label=(\roman*)]
    \item\label{item:GAAA-item1} Its central charge satisfies $Z_{\sigma_1\boxtimes\cdots\boxtimes\sigma_n}(E_1\boxtimes\cdots\boxtimes E_n)=\prod_{i=1}^nZ_i(E_i)$ for all objects $E_1,\ldots,E_n\in\Db(\bP^1)$.
    \item\label{item:GAAA-item2} If each $E_i\in\Db(\bP^1)$ is $\sigma_{i}$-stable of phase $\phi_i$, then $ E_1\boxtimes\cdots\boxtimes E_n$ is $(\sigma_1\boxtimes\cdots\boxtimes\sigma_n)$-stable of phase $\sum_{i=1}^n\phi_i$.
\end{enumerate}
Moreover, an object $E\in\Db((\bP^1)^n)$ is $(\sigma_1\boxtimes\cdots\boxtimes\sigma_n)$-stable if and only if $E\cong E_1\boxtimes\cdots\boxtimes E_n$ for some $\sigma_i$-stable objects $E_i \in \Db(\bP^1)$.
\end{prop}

\begin{proof}
We proceed by induction on $n$. Suppose the statement holds for $n$. To establish the claim for $(\bP^1)^{n+1}$, consider the semiorthogonal decomposition
$$\Db((\bP^1)^{n+1}) = \left\langle \cD_0 \coloneqq \Db((\bP^1)^{n}) \boxtimes \cO(k_{n+1}), \, \cD_1 \coloneqq \Db((\bP^1)^{n}) \boxtimes \cO(k_{n+1}+1) \right\rangle.$$
Both $\cD_0$ and $\cD_1$ are equivalent to $\Db((\bP^1)^n)$. By the inductive hypothesis, there is a unique product stability condition 
$$
\tau \coloneqq \sigma_1 \boxtimes \cdots \boxtimes \sigma_n \in \Stab\Db((\bP^1)^n)
$$
satisfying the desired properties.
On each component $\cD_i$, we define a stability condition $\tau_i$ induced by $\tau$ and the $(n+1)$-th factor $\sigma_{n+1}$ as follows.
First, act on $\tau$ by the $\bC$-action given by the complex number
$$
\log|Z_{\sigma_{n+1}}(\cO_{k_{n+1}}+i)|+\sqrt{-1}\pi\phi_{\sigma_{n+1}}(\cO(k_{n+1}+i))\in\bC,
$$
then transport the stability condition through the equivalence $\Db((\bP^1)^n)\cong\cD_i$.
The resulting stability conditions $\tau_i\in\Stab\cD_i$ satisfy:
\begin{itemize}
    \item $Z_{\tau_i}(E \boxtimes \cO(k_{n+1}+i)) = Z_{\tau}(E) \cdot Z_{\sigma_{n+1}}(\cO(k_{n+1}+i))$ for any $E\in\Db((\bP^{1})^{n})$.
    \item An object $E\in\Db((\bP^1)^n)$ is $\tau$-stable if and only if $E \boxtimes \cO(k_{n+1}+i)$ is $\tau_i$-stable; and in this case, their phases satisfy
    $$
    \phi_{\tau_i}(E \boxtimes \cO(k_{n+1}+i)) = \phi_\tau (E)+\phi_{\sigma_{n+1}}(\cO(k_{n+1}+i)).
    $$
\end{itemize}
We claim that the gluing condition (see Section~\ref{subsubsec:stab-from-SOD})
$$
\Hom(\cA_0,\cA_1[\leq0])=0
$$
holds for the hearts $\cA_i \subseteq \cD_i$ associated with the stability conditions $\tau_i$.
By the induction hypothesis, the stable objects (up to shift) of $\tau$ in $\Db((\bP^1)^n)$ are of the form
$$
\cO(\star)\boxtimes\cL_J
$$
where:
\begin{itemize}
    \item $\star\in\bZ\cup\{\infty\}$; $\cO(\star)\in\Db(\bP^1)$ denotes a line bundle of degree $\star$ when $\star\in\bZ$, or a skyscraper sheaf when $\star=\infty$.
    \item $J=(j_2,\ldots,j_n)\in\{0,1\}^{n-1}$, and $\cL_J=\cO(k_2+j_2)\boxtimes\cdots\boxtimes\cO(k_n+j_n)\in\Db((\bP^1)^{n-1})$.
    We write $\phi_J\coloneqq\phi_{\sigma_2}(\cO(k_2+j_2))+\cdots+\phi_{\sigma_n}(\cO(k_n+j_n))$.
\end{itemize}
Therefore, the heart $\cA_i\subseteq\cD_i$ is generated by appropriate shifts of objects
$$
\cO(\star)\boxtimes\cL_J\boxtimes\cO(k_{n+1}+i).
$$
To prove the gluing condition holds, it suffices to show the vanishing condition between these generators:
$$
\Hom^{\leq0}(\cO(m_0)\boxtimes\cL_J\boxtimes\cO(k_{n+1})[N_0],\cO(m_1)\boxtimes\cL_{J'}\boxtimes\cO(k_{n+1}+1)[N_1])=0
$$
where $N_0,N_1$ are chosen so that the phases (with respect to $\tau_0,\tau_1$, respectively) of the two objects lie in $(0,1]$.
If $J\nleq J'$, no non‑zero morphisms exist. Assume now that $J\leq J'$.
\begin{enumerate}[label=(\roman*)]
    \item If $m_0\leq m_1$, then the morphism from $\cO(m_0)\boxtimes\cL_J\boxtimes\cO(k_{n+1})$ to $\cO(m_1)\boxtimes\cL_{J'}\boxtimes\cO(k_{n+1}+1)$ is concentrated in degree $0$. On the other hand,
\begin{align*}
\phi_{\tau_0}(\cO(m_0)\boxtimes\cL_J\boxtimes\cO(k_{n+1})) &= \phi_{\sigma_1}(\cO(m_0))+\phi_J+\phi_{\sigma_{n+1}}(\cO(k_{n+1})) \\
&\leq\phi_{\sigma_1}(\cO(m_1))+\phi_{J'}+(\phi_{\sigma_{n+1}}(\cO(k_{n+1}+1))-1) \\
& = \phi_{\tau_1}(\cO(m_1)\boxtimes\cL_{J'}\boxtimes\cO(k_{n+1}+1))-1.
\end{align*}
Hence $N_0\geq N_1+1$, and the required vanishing follows.

    \item If $m_0\geq m_1+2$, then the morphism from $\cO(m_0)\boxtimes\cL_J\boxtimes\cO(k_{n+1})$ to $\cO(m_1)\boxtimes\cL_{J'}\boxtimes\cO(k_{n+1}+1)$ is concentrated in degree $1$. On the other hand,
\begin{align*}
\phi_{\tau_0}(\cO(m_0)\boxtimes\cL_J\boxtimes\cO(k_{n+1})) &= \phi_{\sigma_1}(\cO(m_0))+\phi_J+\phi_{\sigma_{n+1}}(\cO(k_{n+1})) \\
&<(\phi_1(\cO(m_1))+1)+\phi_{J'}+(\phi_{n+1}(\cO(k+1))-1) \\
& = \phi_{\tau_1}(\cO(m_1)\boxtimes\cL_{J'}\boxtimes\cO(k_{n+1}+1)).
\end{align*}
Thus $N_0\geq N_1$, and the vanishing again holds.
\end{enumerate}
By Lemma~\ref{lemma:CP-gluing-heart}, the hearts $\cA_i\subseteq\cD_i$ glue to a heart $\cA$ of $\cD$. 

It is easy to check that $0$ is an isolated point of $\text{Im}Z_{\tau_i}(\cA_i)\subseteq\bR_{\geq0}$ for $i=0,1$.
Therefore, Proposition~\ref{prop:CP-gluing-stab-HN} implies that the glued central charge yields a stability function on $\cA$ satisfying the Harder--Narasimhan property.

\medskip\noindent\textbf{Characterization of stable objects.}
The heart $\cA$ is generated by $\cA_0$ and $\cA_1$; therefore, it is generated by appropriate shifts of objects of the form
$$
\cO(\star)\boxtimes\cL_J\boxtimes\cO(k_{n+1}), \quad \cO(\star)\boxtimes\cL_J\boxtimes\cO(k_{n+1}+1).
$$
These objects are stable by \cite[Proposition~2.2]{CPgluing}.
We claim that they are, in fact, the only stable objects.

The proof is similar to the purely algebraic case: for any object $E \in \cA$, we take a filtration whose successive subquotients are the stable generators listed above, and we wish to rearrange it so that the subquotients appear in order of non‑increasing phase.
Recall that every object of $\cA$ admits a filtration whose subquotients lie first in $\cA_1$ and then in $\cA_0$ (see Lemma~\ref{lemma:CP-gluing-heart}).
Hence it is enough to show that whenever
$$
0<\phi(\cO(m_1)\boxtimes\cL_{J'}\boxtimes\cO(k_{n+1}+1)[N_1])<\phi(\cO(m_0)\boxtimes\cL_{J}\boxtimes\cO(k_{n+1})[N_0])\leq1,
$$
we have
$$
\Hom^1(\cO(m_0)\boxtimes\cL_{J}\boxtimes\cO(k_{n+1})[N_0],\cO(m_1)\boxtimes\cL_{J'}\boxtimes\cO(k_{n+1}+1)[N_1])=0.
$$
Again, a necessary condition for a non‑zero extension is $J\leq J'$. There are two cases:
\begin{enumerate}[label=(\roman*)]
    \item If $m_0\leq m_1$, we have
$$
\phi_{\sigma_1}(\cO(m_0))+\phi_J+\phi_{\sigma_{n+1}}(\cO(k_{n+1}))\leq\phi_{\sigma_1}(\cO(m_1))+\phi_{J'}+\phi_{\sigma_{n+1}}(\cO(k_{n+1}+1))-1.
$$
For the extension to be nonzero, we would need $N_0=N_1+1$; but then
$$
\phi(\cO(m_0)\boxtimes\cL_{J}\boxtimes\cO(k)[N_0])\leq\phi(\cO(m_1)\boxtimes\cL_{J'}\boxtimes\cO(k+1)[N_1]),
$$
contradicting the assumed inequality of phases.

    \item If $m_0\geq m_1+2$, we have
$$
\phi_{\sigma_1}(\cO(m_0))+\phi_J+\phi_{\sigma_{n+1}}(\cO(k_{n+1}))<(\phi_{\sigma_1}(\cO(m_1))+1)+\phi_{J'}+(\phi_{\sigma_{n+1}}(\cO(k_{n+1}+1))-1).
$$
For the extension to be nonzero, we would need $N_0=N_1$; but then
$$
\phi(\cO(m_0)\boxtimes\cL_{J}\boxtimes\cO(k)[N_0])<\phi(\cO(m_1)\boxtimes\cL_{J'}\boxtimes\cO(k+1)[N_1]),
$$
again contradicting the phase ordering.
\end{enumerate}
Thus the desired extensions vanish, and the claimed description of stable objects follows.

\medskip\noindent\textbf{Support property.}
Each factor stability condition satisfies the support property
$$
|Z_i(E_i)|\geq C_i||E_i|| \quad \text{ for every }\sigma_i\text{-stable object } E_i.
$$
Since every product‑stable object splits as $E_1\boxtimes\cdots\boxtimes E_n$, we have
$$
|Z(E_1\boxtimes\cdots\boxtimes E_n)|=\prod_{i=1}^n|Z_i(E_i)|\geq\prod_{i=1}^n C_i \prod_{i=1}^n||E_i||=\left(\prod_{i=1}^n C_i\right)||E_1\boxtimes\cdots\boxtimes E_n||.
$$
Hence the support property holds for the glued stability condition.

\medskip\noindent\textbf{Uniqueness.}
Let $\sigma$ be a stability condition satisfying Conditions~\ref{item:GAAA-item1} and \ref{item:GAAA-item2}.
Its central charge is uniquely determined by the product formula. Moreover, the heart is generated by the stable objects of phase in $(0,1]$; these are precisely appropriate shifts of objects of the form $\cO(\star)\boxtimes\cL_J\boxtimes\cO(k_{n+1}+i)$. Therefore, the heart must coincide with the heart $\cA$ constructed above. 
Thus $\sigma=(Z,\cA)$ is uniquely determined.
\end{proof}

We have now shown that for every $(\sigma_1,\ldots,\sigma_n)\in\cS_n$, there exists a unique product-type stability condition $\sigma_1\boxtimes\cdots\boxtimes\sigma_n\in\Db((\bP^1)^n)$, and that it satisfies Property~\ref{item-property-a}. It remains to prove the following lemma.

\begin{lemma}\label{lemma:property-b}
Let $(\sigma_1,\ldots,\sigma_n),(\sigma'_1,\ldots,\sigma'_n)\in\cS_n$. Then
$$
\overline{\sigma_1\boxtimes\cdots\boxtimes\sigma_n}=\overline{\sigma'_1\boxtimes\cdots\boxtimes\sigma'_n}\in\Stab\Db((\bP^1)^n)/\bC
$$
if and only if $(\overline{\sigma_1},\ldots,\overline{\sigma_n})=(\overline{\sigma'_1},\ldots,\overline{\sigma'_n})\in\overline\cS_n$.
\end{lemma}

\begin{proof}
The ``if" direction is straightforward. 
For the converse, suppose
$$
\overline{\sigma_1\boxtimes\cdots\boxtimes\sigma_n}=\overline{\sigma'_1\boxtimes\cdots\boxtimes\sigma'_n}\in\Stab\Db((\bP^1)^n)/\bC.
$$
Comparing the sets of stable objects shows that for each $i$ the stability conditions $\sigma_i$ and $\sigma_i'$ have the same stable objects.
Choose $\cO(k_i),\cO(k_i+1)\in\Db(\bP^1)$ that are $\sigma_i$-stable (hence also $\sigma_i'$-stable).
Then $\sigma_i$ is uniquely determined by the following data:
$$
    \phi_{\sigma_i}(\cO(k_i))=\psi_i,\quad
    \phi_{\sigma_i}(\cO(k_i+1))=\psi_i+\phi_i,\quad 
    |Z_{\sigma_i}(\cO(k_i))|=m_{i,0}, \quad
    |Z_{\sigma_i}(\cO(k_i+1))|=m_{i,1}.
$$
Similarly, $\sigma_i'$ is uniquely determined by a quadruple $(\psi_i'.\phi_i',m_{i,0}',m_{i,1}')$.
Note that $\overline{\sigma_i}=\overline{\sigma_i'}$ if and only if $\phi_i=\phi_i'$ and $\frac{m_{i,1}}{m_{i,0}}=\frac{m'_{i,1}}{m'_{i,0}}$.

For $I=(i_1,\ldots,i_n)\in\{0,1\}^n$, denote $\cL_I=\cO(k_1+i_1,\ldots,k_n+i_n)$.
With respect to $\sigma_1\boxtimes\cdots\boxtimes\sigma_n$, the object $\cL_I$ is stable with
$$
\phi(\cL_I)=\left(\sum_{\ell=1}^n\psi_\ell\right)+\left(\sum_{\ell=1}^n i_\ell\phi_\ell\right) \quad \text{ and } \quad
|Z(\cL_I)|=\left(\prod_{\ell=1}^nm_{\ell,0}\right)\left(\prod_{\ell=1}^n\frac{m_{\ell,i_\ell}}{m_{i,0}}\right).
$$
Similarly, $\cL_I$ is $(\sigma_1'\boxtimes\cdots\boxtimes\sigma_n')$-stable with
$$
\phi'(\cL_I)=\left(\sum_{\ell=1}^n\psi'_\ell\right)+\left(\sum_{\ell=1}^n i_\ell\phi'_\ell\right) \quad \text{ and } \quad
|Z'(\cL_I)|=\left(\prod_{\ell=1}^nm'_{\ell,0}\right)\left(\prod_{\ell=1}^n\frac{m'_{\ell,i_\ell}}{m'_{i,0}}\right).
$$
Since the two product stability condition lies in the same $\bC$-orbit, for every $I=(i_1,\ldots,i_n)\in\{0,1\}^n$, we have
$$
\sum_{\ell=1}^n i_\ell\phi_\ell=\sum_{\ell=1}^n i_\ell\phi_\ell' \quad \text{ and } \quad
\prod_{\ell=1}^n\frac{m_{\ell,i_\ell}}{m_{i,0}}=\prod_{\ell=1}^n\frac{m_{\ell,i_\ell}'}{m_{i,0}'}.
$$
Taking $I$ with a single non‑zero entry at the $i$-th coordinate yields $\phi_i=\phi_i'$ and $\frac{m_{i,1}}{m_{i,0}}=\frac{m'_{i,1}}{m'_{i,0}}$ for each $i$.
Therefore,
$$
(\overline{\sigma_1},\ldots,\overline{\sigma_n})=(\overline{\sigma'_1},\ldots,\overline{\sigma'_n})\in\overline\cS_n,
$$
which completes the proof.
\end{proof}

\begin{proof}[Proof of Theorem~\ref{thm:main-AG-Sn}]
The existence and uniqueness of product-type stability condition for each $(\sigma_1,\ldots,\sigma_n)\in\cS_n$, together with Property~\ref{item-property-a}, follow from Corollary~\ref{cor:ALG-EUa} (the purely algebraic case) and Proposition~\ref{prop:GeoAlg-EUa} (the case with one geometric factor). Property~\ref{item-property-b} is proved in Lemma~\ref{lemma:property-b}.
This completes the proof of Theorem~\ref{thm:main-AG-Sn}.
\end{proof}

\begin{rmk}\label{rmk:Sn-not-geometric}
By Theorem~\ref{thm:main-AG-Sn}\ref{item-property-a}, the product-type stability conditions on $\Db((\bP^1)^n)$ obtained from $\cS_n$ are \emph{not geometric} for $n\geq2$: 
skyscraper sheaves on $(\bP^1)^n$ are not stable with respect to any of them.
\end{rmk}

\begin{rmk}\label{rmk:SnxC->P^1n}
Recall from Remark~\ref{rmk:slice-C-Stab(P1)}\ref{item:slice-C-Stab(P1)} that, after fixing a homological mirror symmetry equivalence 
$$
\Mir\colon\Db(\bP^1)\xrightarrow{\sim} H^0\Tw\cW(\bC^\times,\fs),
$$
we may choose the slice $\Stab^{c=0}\Db(\bP^1)$ of the $\bC$-action
$$
\Stab^{c=0}\Db(\bP^1) \subseteq \Stab\Db(\bP^1),
$$
on which the central charges take the form
$$
Z_a(E)=\int_{[\Mir(E)]}\exp\left(z+\frac{e^a}{z}\right)\frac{\dd z}{z}.
$$
We can therefore use an $n$-tuple $(a_1,\ldots,a_n)\in\bC^n$ to parametrize elements of 
$$
\overline{\cS}_n\subseteq(\Stab^{c=0}\Db(\bP^1))^n\subseteq(\Stab\Db(\bP^1))^n.
$$
Together with the free $\bC$-action yields an embedding
$$
\overline{\cS}_n\times\bC\hookrightarrow(\Stab\Db(\bP^1))^n; \qquad 
\left((\sigma_{a_1},\ldots,\sigma_{a_n}),c\right)\mapsto(\sigma_{a_1}\boxtimes\cdots\boxtimes\sigma_{a_n})\cdot c.
$$
\end{rmk}

\subsection{Stable implies special Lagrangian.}
\label{subsection:Pf-of-Main-SG}
We now prove Theorem~\ref{thm:IntroMainThmSG}, restated as follows:

\begin{thm}\label{thm:SG-main-thm}
For the product-type stability conditions constructed in Theorem~\ref{thm:main-AG-Sn},
$$
\overline{\cS}_n\times\bC\subseteq\Stab\Db((\bP^1)^n)\cong\Stab H^0\Tw(\cW((\bC^\times)^n,\fs_n)), 
$$
where $\overline{\cS}_n\subseteq\bC^n$ is parametrized by $(a_1,\ldots,a_n)$ as in Remark~\ref{rmk:SnxC->P^1n}.
Write
$$
(a_1,\ldots,a_n,c)\mapsto\sigma_{\mathbf{a},c}=(Z_{\mathbf{a},c},\cP_{\mathbf{a},c})\in\Stab H^0\Tw(\cW((\bC^\times)^n,\fs_n)).
$$
Then:
\begin{enumerate}[label=(\alph*)]
    \item\label{item:SGthm-item-1} The central charge $Z_{\mathbf{a},c}$ is given by the integral
$$
Z_{\mathbf{a},c}(\bL)=\int_{[\bL]}\exp\left(z_1+\cdots+z_n+c+\frac{e^{a_1}}{z_1}+\cdots+\frac{e^{a_n}}{z_n}\right)\frac{\dd z_1}{z_1}\cdots\frac{\dd z_n}{z_n}.
$$
    \item\label{item:SGthm-item-2} If an object $\bL$ is $\sigma_{\mathbf{a},c}$-stable, then it admits a special Lagrangian representative (up to isotopy) in $((\bC^\times)^n,\omega_n)$ with respect to the holomorphic $n$-form
$$
\Omega_{\mathbf{a},c}\coloneqq\exp\left(z_1+\cdots+z_n+c+\frac{e^{a_1}}{z_1}+\cdots+\frac{e^{a_n}}{z_n}\right)\frac{\dd z_1}{z_1}\cdots\frac{\dd z_n}{z_n}.
$$
\end{enumerate}
\end{thm}

\begin{proof}[Proof of \ref{item:SGthm-item-1}]
It suffices to verify the central-charge formula on objects of the form $\Mir_n(\cO(k_1,\ldots,k_n))$ with $k_1,\ldots,k_n\in\bZ$, as they generate the numerical Grothendieck group of the category.
Denote $L_{k_i}$ the mirror Lagrangian of $\cO(k_i)\in\Db(\bP^1)$. Then we have:
\begin{align*}
Z_{\mathbf{a},c}(\Mir_n(\cO(k_1,\ldots,k_n))) &= Z_{(\sigma_{a_1}\boxtimes\cdots\boxtimes\sigma_{a_n})\cdot c}(\cO(k_1,\ldots,k_n)) \\
&= e^c \prod_{i=1}^n Z_{\sigma_{a_i}}(\cO(k_i)) \\
& = e^c \prod_{i=1}^n \int_{[L_{k_i}]}\exp\left(z_i+\frac{e^{a_i}}{z_i}\right)\frac{\dd z_i}{z_i} \\
& = \int_{[L_{k_1}]\times\cdots\times[L_{k_n}]}\exp\left(z_1+\cdots+z_n+c+\frac{e^{a_1}}{z_1}+\cdots+\frac{e^{a_n}}{z_n}\right)\frac{\dd z_1}{z_1}\cdots\frac{\dd z_n}{z_n} \\
& = \int_{[\Mir_n(\cO(k_1,\ldots,k_n))]}\exp\left(z_1+\cdots+z_n+c+\frac{e^{a_1}}{z_1}+\cdots+\frac{e^{a_n}}{z_n}\right)\frac{\dd z_1}{z_1}\cdots\frac{\dd z_n}{z_n}.
\end{align*}
The last equality follows from the fact that $\Mir_n(\cO(k_1,\ldots,k_n))$ admits a representative that is isotopic to the product $L_{k_1}\times\cdots\times L_{k_n}$ (see Lemma~\ref{lemma:Mirror-(P1^)n-H0TwW}). This proves \ref{item:SGthm-item-1}.
\end{proof}

\begin{proof}[Proof of \ref{item:SGthm-item-2}]
An object $\bL$ is $\sigma_{\mathbf{a},c}$-stable if and only if $\Mir_n^{-1}(\bL)$ is $(\sigma_{a_1}\boxtimes\cdots\boxtimes\sigma_{a_n})$-stable, which by Theorem~\ref{thm:main-AG-Sn}\ref{item-property-a} holds if and only if 
$$
\Mir_n^{-1}(\bL)\cong E_1\boxtimes\cdots\boxtimes E_n
$$
where each $E_i$ is $\sigma_{a_i}$-stable. By Corollary~\ref{cor:StabDbP^1viaHKK}, $\Mir(E_i)$ admits a special Lagrangian representative with respect to the holomorphic $1$-form
$$
\exp\left(z+\frac{e^{a_i}}{z}\right)\frac{\dd z}{z}.
$$
On the other hand, we have
$$
\bL\cong\Mir_n(E_1\boxtimes\cdots\boxtimes E_n),
$$
which admits a representative isotopic to the product
$$
\Mir(E_1)\times\cdots\times\Mir(E_n).
$$
This is a Lagrangian submanifold of $((\bC^\times)^n,\omega_n)$ because the symplectic form $\omega_n$ is split, and it is a special Lagrangian submanifold with respect to the holomorphic $n$-form
\begin{align*}
\Omega_{\mathbf{a},c} &=\exp\left(z_1+\cdots+z_n+c+\frac{e^{a_1}}{z_1}+\cdots+\frac{e^{a_n}}{z_n}\right)\frac{\dd z_1}{z_1}\cdots\frac{\dd z_n}{z_n} \\
&=\exp(c)\cdot\prod_{i=1}^n \left(\exp\left(z_i+\frac{e^{a_i}}{z_i}\right)\frac{\dd z_i}{z_i}\right),
\end{align*}
since each factor $\Mir(E_i)$ is of constant phase with respect to $\exp\left(z+\frac{e^{a_i}}{z}\right)\frac{\dd z}{z}$.
This completes the proof.
\end{proof}

\section{More product-type stability conditions.}
This section contains the proof of Theorem~\ref{thm:embedding-P1-E}.
The embedding via product stability
$$
(\Stab\Db(\bP^1)/\bC)\times (\Stab\Db(\bP^1)/\bC)\hookrightarrow\Stab\Db(\bP^1\times\bP^1)/\bC
$$
is established in Section~\ref{subsec:last-P1xP1}.
The analogous embedding for two elliptic curves $X_1,X_2$,
$$
(\Stab\Db(X_1)/\bC)\times (\Stab\Db(X_2)/\bC)\hookrightarrow\Stab\Db(X_1\times X_2)/\bC,
$$
is treated in Section~\ref{subsec:last-ExE}.

\subsection{Product-type stability conditions on \texorpdfstring{$\bP^1\times\bP^1$}{}.}
\label{subsec:last-P1xP1}

\begin{thm}
\label{thm:P1xP1}
For every $(\sigma_1,\sigma_2)\in(\Stab\Db(\bP^1))^2$, there exists a unique stability condition, denoted by $\sigma_1\boxtimes\sigma_2\in\Stab\Db(\bP^1\times\bP^1)$, such that:
\begin{enumerate}[label=(\roman*)]
\item\label{item:p1xp1thm-item1} Its central charge satisfies
$$
Z_{\sigma_1\boxtimes\sigma_2}(E_1\boxtimes E_2)=Z_{\sigma_1}(E_1)Z_{\sigma_2}(E_2)
$$
for all objects $E_1,E_2\in\Db(\bP^1)$.
\item\label{item:p1xp1thm-item2} If each $E_i\in\Db(\bP^1)$ is $\sigma_i$-stable of phase $\phi_i$, then $E_1\boxtimes E_2$ is $(\sigma_1\boxtimes\sigma_2)$-stable of phase $\phi_1+\phi_2$.
\end{enumerate}
Moreover, these product-type stability conditions satisfy:
\begin{enumerate}[label=(\alph*)]
\item\label{item:p1xp1thm-itema} An object $E_1\boxtimes E_2$ is $(\sigma_1\boxtimes\sigma_2)$-stable if and only if $E_i$ is $\sigma_i$-stable for both $i=1,2$.
\item\label{item:p1xp1thm-itemb} Two elements $(\sigma_1,\sigma_2),(\sigma'_1,\sigma'_2)\in(\Stab\Db(\bP^1))^2$ yield the same product stability condition up to the $\bC$-action, i.e.,
$$
\overline{\sigma_1\boxtimes\sigma_2}=\overline{\sigma'_1\boxtimes\sigma'_2}\in\Stab\Db(\bP^1\times\bP^1)/\bC,
$$
if and only if $(\overline{\sigma_1},\overline{\sigma_2})=(\overline{\sigma'_1},\overline{\sigma'_2})$ in $(\Stab\Db(\bP^1)/\bC)^2$.
\end{enumerate}
Consequently, there is a well-defined embedding
$$
(\Stab\Db(\bP^1)/\bC)\times (\Stab\Db(\bP^1)/\bC)\hookrightarrow\Stab\Db(\bP^1\times\bP^1)/\bC.
$$
\end{thm}

The existence and uniqueness of product-type stability conditions of $(\sigma_1,\sigma_2)\in(\Stab\Db(\bP^1))^2$ has been established in Theorem~\ref{thm:main-AG-Sn} when at most one of the $\sigma_i$ is geometric. It therefore remains to treat the case where both factors are geometric:
$$
\sigma_1,\sigma_2\in\Stab^\Geo\Db(\bP^1).
$$
If such a product stability condition exists, then skyscraper sheaves on $\bP^1\times\bP^1$ must be stable; so we are looking for \emph{geometric} stability conditions on $\Db(\bP^1\times\bP^1)$.

Geometric stability conditions on smooth projective surfaces have been constructed in \cite{BriK3,ArcaraBertram}. 
Given $B,H\in H^{1,1}(\bP^1\times\bP^1,\bR)$ with $H$-ample, there exists a unique geometric stability condition $\sigma_{B,H}$ on $\Db(\bP^1\times\bP^1)$ wtih central charge
$$
Z_{B,H}(E)=-\int_{\bP^1\times\bP^1} e^{-(B+\sqrt{-1}H)}\ch(E)
$$
and for which skyscraper sheaves are stable of phase $1$.

We observe that this central charge admits a product decomposition:

\begin{lemma}
\label{lemma:Z-product-geometric}
Let $D_1,D_2$ be the generator of $H^{1,1}(\bP^1\times\bP^1,\bZ)$ corresponding to the two fiber directions. Write
$$
B=b_1D_1+b_2D_2, \quad H=h_1D_1+h_2D_2, \quad h_1,h_2>0.
$$
Then
$$
Z_{B,H}(E_1\boxtimes E_2)=-Z_1(E_1)Z_2(E_2)
$$
where 
$$
Z_i(E_i)\coloneqq -\deg(E_i)+\left(b_i+\sqrt{-1}h_i\right)\rank(E_i).
$$
\end{lemma}

\begin{proof}
It suffices to verify the equality for line bundles $E_1\boxtimes E_2=\cO(k_1,k_2)$ with $k_1,k_2\in\bZ$:
\begin{align*}
Z_{B,H}(\cO(k_1,k_2)) & = -\int e^{-(B+\sqrt{-1}H)}e^{c_1(\cO(k_1,k_2))} \\
& = -\frac{1}{2} \left(c_1(\cO(k_1,k_2))-B-\sqrt{-1}H\right)^2 \\
& = -\frac{1}{2}c_1\left(\cO(k_1,k_2)\right)^2+c_1(\cO(k_1,k_2)).\left(B+\sqrt{-1}H\right)-\frac{1}{2}\left(B+\sqrt{-1}H\right)^2 \\
& = -k_1k_2 + k_1\left(b_2+\sqrt{-1}h_2\right) + k_2\left(b_1+\sqrt{-1}h_1\right) - \left(b_1b_2 - h_1h_2 + \sqrt{-1} (b_1h_2+b_2h_1)\right) \\
& = - \left(-k_1+b_1+\sqrt{-1}h_1\right)\left(-k_2+b_2+\sqrt{-1}h_2\right) \\
& = -Z_1(\cO(k_1))Z_2(\cO(k_2)).
\end{align*}
\end{proof}

Recall that every geometric stability condition on $\bP^1$ can be written as $\sigma_\tau\cdot c$ for some $\tau\in\bH$ and $c\in\bC$, where $\sigma_\tau=(Z_\tau,\cP_\tau)$ is the unique stability condition on $\bP^1$ such that:
\begin{itemize}
    \item $Z_\tau(E)=-\deg(E)+\tau\cdot\rank(E)$ for every $E\in\Db(\bP^1)$.
    \item Skyscraper sheaves are $\sigma_\tau$-stable and of phase $1$.
\end{itemize}

\begin{prop}
\label{prop:GeomxGeom}
Let $\sigma_1,\sigma_2\in\Stab^\Geo\Db(\bP^1)$ be two geometric stability conditions. Write
$$
\sigma_1=\sigma_{\tau_1}\cdot c_1,  \qquad
\sigma_2=\sigma_{\tau_2}\cdot c_2.
$$
for some $\tau_1,\tau_2\in\bH$ and $c_1,c_2\in\bC$. 
Define
$$
B=\text{Re}(\tau_1)D_1+\text{Re}(\tau_2)D_2, \qquad    
H=\text{Im}(\tau_1)D_1+\text{Im}(\tau_2)D_2.
$$
Then
$$
\sigma_{B,H}\cdot\left(c_1+c_2+\sqrt{-1}\pi\right)
$$
is the unique stability condition on $\Db(\bP^1\times\bP^1)$ that satisfies the product‑type conditions~\ref{item:p1xp1thm-item1} and \ref{item:p1xp1thm-item2} of Theorem~\ref{thm:P1xP1}.
\end{prop}

\begin{proof}
By Lemma~\ref{lemma:Z-product-geometric}, the central charge of $\sigma_{B,H}\cdot\left(c_1+c_2+\sqrt{-1}\pi\right)$ satisfies the product condition \ref{item:p1xp1thm-item1}.
To verify the second condition, we need to show that each of the following objects
$$
\cO(k_1,k_2), \quad \cO(k)\boxtimes \cO_x, \quad \cO_x\boxtimes\cO(k), \quad \cO_{(x_1,x_2)}
$$
is $\sigma_{B,H}\cdot\left(c_1+c_2+\sqrt{-1}\pi\right)$-stable and has the desired phase.

\medskip\noindent\textbf{Skyscraper sheaves $\cO_{(x_1,x_2)}$:}
Skyscraper sheaves are stable with respect to any geometric stability conditions; their phases satisfy:
\begin{align*}
\phi_1(\cO_{x_1})+\phi_2(\cO_{x_2}) & =\left(\frac{1}{\pi}\text{Im}(c_1)+1\right)+\left(\frac{1}{\pi}\text{Im}(c_2)+1\right) \\
&= \frac{1}{\pi}\text{Im}\left(c_1+c_2+\sqrt{-1}\pi\right)+1 \\
&= \phi_{\sigma_{B,H}\cdot\left(c_1+c_2+\sqrt{-1}\pi\right)}(\cO_{(x_1,x_2)}).
\end{align*}

\medskip\noindent\textbf{Line bundles $\cO(k_1,k_2)$:}
The stability of line bundles on $\bP^1\times\bP^1$ with respect to geometric stability conditions follows from \cite[Theorem~1.1]{ArcaraMiles}, as there are no curves in $\bP^1\times\bP^1$ of negative self-intersection.
It remains to show that their phases satisfy:
$$
\phi_1(\cO(k_1))+\phi_2(\cO(k_2))=\phi_{\sigma_{B,H}\cdot\left(c_1+c_2+\sqrt{-1}\pi\right)}(\cO(k_1,k_2)).
$$
Since the central charge is of product-type, we have
$$
\phi_1(\cO(k_1))+\phi_2(\cO(k_2))-\phi_{\sigma_{B,H}\cdot\left(c_1+c_2+\sqrt{-1}\pi\right)}(\cO(k_1,k_2))\in2\bZ.
$$
The phase of $\cO(k_i)$ with respect to $\sigma_i$ lies in $\left(\frac{1}{\pi}\text{Im}(c_i),\frac{1}{\pi}\text{Im}(c_i)+1\right)$, thus
$$
\phi_1(\cO(k_1))+\phi_2(\cO(k_2))\in\left(\frac{1}{\pi}\text{Im}(c_1+c_2),\frac{1}{\pi}\text{Im}(c_1+c_2)+2\right).
$$
On the other hand, we have $\phi_{\sigma_{B,H}}(\cO(k_1,k_2))\in(-1,1)$ by \cite[Lemma~10.1]{BriK3}, thus
$$
\phi_{\sigma_{B,H}\cdot\left(c_1+c_2+\sqrt{-1}\pi\right)}(\cO(k_1,k_2))\in\left(\frac{1}{\pi}\text{Im}(c_1+c_2),\frac{1}{\pi}\text{Im}(c_1+c_2)+2\right).
$$
This proves that the equality of phases holds.

\medskip\noindent\textbf{Torsion sheaves $\cO(k)\boxtimes\cO_x$:}
Recall that the heart $\cA_{B,H}$ of the stability condition $\sigma_{B,H}$ is given by the tilting construction: Consider the torsion pair
$$
\cT_{B,H}=\left<\text{torsion sheaves}, \mu_H\text{-stable sheaves } E \text{ with } \mu_H(E)>\frac{B.H}{H^2}\right>,
$$
$$
\cF_{B,H}=\left<\mu_H\text{-stable sheaves } E \text{ with } \mu_H(E)\leq\frac{B.H}{H^2}\right>.
$$
Then the heart $\cA_{B,H}$ is given by
$$
\cA_{B,H}=\left<\cT_{B,H},\cF_{B,H}[1]\right>.
$$

Let $\iota\colon\bP^1\times\{pt\}\hookrightarrow\bP^1\times\bP^1$ be the inclusion of the first factor. We claim that for any $k\in\bZ$,
$$
\iota_*\cO(k)\cong\cO(k)\boxtimes\cO_x
$$
is $\sigma_{B,H}$-stable.
The argument is similar to \cite[Lemma~6.3]{BriK3}.
Observe that the object $\iota_*\cO(k)\in\cT_{B,H}\subseteq\cA_{B,H}$ lies in the heart.
Suppose there is a short exact sequence in $\cA_{B,H}$
$$
0\to A\to \iota_*\cO(k)\to B\to 0 .
$$
Taking cohomology yields
$$
0\to H^{-1}(B) \to H^0(A) \to \iota_*\cO(k) \to H^0(B) \to 0,
$$
where $H^{-1}(B)$ is torsion-free by \cite[Lemma~10.1]{BriK3}.
Therefore, the $\mu_H$-semistable factors of $H^{-1}(B)$ and $H^0(A)$ have the same slopes.
This contradicts with the definition of $\cA_{B,H}$ unless $H^{-1}(B)=0$.
Thus, both $A,B\in T$ are torsion sheaves and the original sequence is an exact sequence of sheaves.

Proper subobjects of $\iota_*\cO(k)$ are of the form $\iota_*\cO(\ell)$ with $\ell<k$; it remains to check that they have strictly smaller phases.
\begin{align*}
Z_{B,H}(\iota_*\cO(\ell)) &= -\ell + (B+iH).[\bP^1\times\{pt\}] \\
& = -\ell + (b_2+ih_2).
\end{align*}
When $\ell<k$, the phase of $\iota_*\cO(\ell)$ is strictly smaller than that of $\iota_*\cO(k)$.
Hence $\iota_*\cO(k)$ is $\sigma_{B,H}$-stable.

Finally, we verify that $\cO(k)\boxtimes\cO_x$ has the desired phase. Again, because the central charge is of product-type, we have
$$
\phi_1(\cO(k))+\phi_2(\cO_x)-\phi_{\sigma_{B,H}\cdot\left(c_1+c_2+\sqrt{-1}\pi\right)}(\cO(k)\boxtimes\cO_x)\in2\bZ.
$$
We have
$$
\phi_1(\cO(k))+\phi_2(\cO_x)\in\left(\frac{1}{\pi}\text{Im}(c_1+c_2)+1,\frac{1}{\pi}\text{Im}(c_1+c_2)+2\right).
$$
By \cite[Lemma~10.1]{BriK3}, the torsion sheaf has phase $\phi_{\sigma_{B,H}}(\cO(k)\boxtimes\cO_x)\in(0,1)$, thus
$$
\phi_{\sigma_{B,H}\cdot\left(c_1+c_2+\sqrt{-1}\pi\right)}(\cO(k)\boxtimes\cO_x)\in\left(\frac{1}{\pi}\text{Im}(c_1+c_2)+1,\frac{1}{\pi}\text{Im}(c_1+c_2)+2\right).
$$
This concludes the proof that the stability condition $\sigma_{B,H}\cdot\left(c_1+c_2+\sqrt{-1}\pi\right)$ satisfies Conditions~\ref{item:p1xp1thm-item1} and \ref{item:p1xp1thm-item2}.

\medskip\noindent\textbf{Uniqueness.}
Geometric stability conditions are uniquely determined by their central charge together with the phases of skyscraper sheaves.
The product‑type conditions fix both, so the constructed stability condition is unique.
\end{proof}

\begin{proof}[Proof of Theorem~\ref{thm:P1xP1}]
The existence and uniqueness of the product stability conditions follow from Theorem~\ref{thm:main-AG-Sn} (the case where at most one factor is geometric) and Proposition~\ref{prop:GeomxGeom} (the purely geometric case).

\medskip\noindent\textbf{Part~\ref{item:p1xp1thm-itema}.}
The only case not treated by Theorem~\ref{thm:main-AG-Sn} occurs when $\sigma_1,\sigma_2$ are both geometric.
Suppose $\sigma_1,\sigma_2\in\Stab^\Geo\Db(\bP^1)$ and that $E_1\boxtimes E_2$ is $(\sigma_1\boxtimes\sigma_2)$-stable.
We need to show that each $E_i$ is a line bundle or a skyscraper sheaf (up to shift).
If either $E_1$ or $E_2$ were not of this form, then $E_1\boxtimes E_2$ would be decomposable, contradicting its stability.

\medskip\noindent\textbf{Part~\ref{item:p1xp1thm-itemb}.}
Assume that two product‑type stability conditions coincide up to the $\bC$-action:
$$
\overline{\sigma_1\boxtimes\sigma_2}=\overline{\sigma'_1\boxtimes\sigma'_2}\in\Stab\Db(\bP^1\times\bP^1)/\bC.
$$
Comparing the sets of stable objects shows that for each factor the two stability conditions must be of the same type: either both geometric or both algebraic.
The only case not covered by Theorem~\ref{thm:main-AG-Sn} is when all four stability conditions $\sigma_1,\sigma_1',\sigma_2,\sigma_2'$ are geometric.
Write
$$
\sigma_i=\sigma_{\tau_i}\cdot c_i, \qquad \sigma_i'=\sigma_{\tau_i'}\cdot c_i'
$$
for $\tau_i,\tau_i'\in\bH$ and $c_i,c_i'\in\bC$ as in Proposition~\ref{prop:GeomxGeom}.
By the same proposition, we have
$$
\sigma_1\boxtimes\sigma_2=\sigma_{B,H}\cdot\left(c_1+c_2+\sqrt{-1}\pi\right) \quad \text{ and } \quad
\sigma_1'\boxtimes\sigma_2'=\sigma_{B',H'}\cdot\left(c_1'+c_2'+\sqrt{-1}\pi\right)
$$
where
$$
B=\text{Re}(\tau_1)D_1+\text{Re}(\tau_2)D_2, \qquad    
B'=\text{Re}(\tau_1')D_1+\text{Re}(\tau_2')D_2,
$$
$$
H=\text{Im}(\tau_1)D_1+\text{Im}(\tau_2)D_2, \qquad    
H=\text{Im}(\tau_1')D_1+\text{Im}(\tau_2')D_2.
$$
The assumption $\overline{\sigma_1\boxtimes\sigma_2}=\overline{\sigma'_1\boxtimes\sigma'_2}$ is equivalent to 
$$
B=B' \quad \text{ and } \quad H=H'
$$
which in turn is equivalent to
$$
\tau_1=\tau_1' \quad \text{ and } \quad \tau_2=\tau_2'.
$$
Hence $\overline{\sigma_i}=\overline{\sigma_i'}$ for $i=1,2$, completing the proof.
\end{proof}

\begin{rmk}
\label{rmk:Diagonal_stable}
The argument used for torsion sheaves above also shows that $\Delta_*\cO(m)$ is $(\sigma_1\boxtimes\sigma_2)$-stable for $\sigma_1,\sigma_2\in\Stab^\Geo\Db(\bP^1)$, where $\Delta\colon\bP^1\hookrightarrow\bP^1\times\bP^1$ is the diagonal embedding.
Therefore, in the geometric case, not every stable object splits as a product $E_1\boxtimes E_2$, in contrast with the situation in Theorem~\ref{thm:main-AG-Sn} (where at most one factor is geometric).

It would be interesting to investigate whether the mirror object $\Mir_2(\Delta_*\cO(m))$ admits a special Lagrangian representative with respect to the holomorphic $2$-form associated to $\sigma_1\boxtimes\sigma_2$.
\end{rmk}

\subsection{Product-type stability conditions on products of two elliptic curves.}
\label{subsec:last-ExE}

\begin{thm}
\label{thm:E1xE2}
Let $X_1$ and $X_2$ be two elliptic curves.
For every $(\sigma_1,\sigma_2)\in\Stab\Db(X_1)\times\Stab\Db(X_2)$, there exists a unique stability condition, denoted by $\sigma_1\boxtimes\sigma_2\in\Stab\Db(X_1\times X_2)$, such that:
\begin{enumerate}[label=(\roman*)]
\item Its central charge satisfies
$$
Z_{\sigma_1\boxtimes\sigma_2}(E_1\boxtimes E_2)=Z_{\sigma_1}(E_1)Z_{\sigma_2}(E_2)
$$
for all objects $E_i\in\Db(X_i)$.
\item If each $E_i\in\Db(X_i)$ is $\sigma_i$-stable of phase $\phi_i$, then $E_1\boxtimes E_2$ is $(\sigma_1\boxtimes\sigma_2)$-stable of phase $\phi_1+\phi_2$.
\end{enumerate}
Moreover, these product-type stability conditions satisfy:
\begin{enumerate}[label=(\alph*)]
\item\label{item:ExE-item-a} An object $E_1\boxtimes E_2$ is $(\sigma_1\boxtimes\sigma_2)$-stable if and only if $E_i$ is $\sigma_i$-stable for both $i=1,2$.
\item\label{item:ExE-item-b} Two elements $(\sigma_1,\sigma_2),(\sigma'_1,\sigma'_2)\in\Stab\Db(X_1)\times\Stab\Db(X_2)$ yield the same product stability condition up to the $\bC$-action, i.e.,
$$
\overline{\sigma_1\boxtimes\sigma_2}=\overline{\sigma'_1\boxtimes\sigma'_2}\in\Stab\Db(X_1\times X_2)/\bC,
$$
if and only if $(\overline{\sigma_1},\overline{\sigma_2})=(\overline{\sigma'_1},\overline{\sigma'_2})$ in $(\Stab\Db(X_1)/\bC)\times(\Stab\Db(X_2)/\bC)$.
\end{enumerate}
Consequently, there is a well-defined embedding
$$
(\Stab\Db(X_1)/\bC)\times (\Stab\Db(X_2)/\bC)\hookrightarrow\Stab\Db(X_1\times X_2)/\bC.
$$
\end{thm}

\begin{proof}
All stability conditions on elliptic curves are geometric. Write
$$
\sigma_1=\sigma_{\tau_1}\cdot c_1 \quad \text{ and } \quad
\sigma_2=\sigma_{\tau_2}\cdot c_2 \quad (\tau_i\in\bH,\; c_i\in\bC)
$$
as in Proposition~\ref{prop:GeomxGeom}.
Let $D_i$ be the divisors corresponding to the fibres of $X_1\times X_2$, and set
$$
B=\text{Re}(\tau_1)D_1+\text{Re}(\tau_2)D_2, \qquad    
H=\text{Im}(\tau_1)D_1+\text{Im}(\tau_2)D_2.
$$
Then the geometric stability condition
$$
\sigma_{B,H}\cdot\left(c_1+c_2+\sqrt{-1}\pi\right)\in\Stab\Db(X_1\times X_2)
$$
has the desired central charge (see Lemma~\ref{lemma:Z-product-geometric}).
We need to show that the objects
$$
E_1\boxtimes E_2, \quad E\boxtimes\cO_x, \quad \cO_x\boxtimes E, \quad \cO_{(x_1,x_2)},
$$
where each $E$ is a slope‑stable vector bundle on the corresponding elliptic curve, are $\sigma_{B,H}$-stable and have the correct phases.
The phase computation is identical to that in Proposition~\ref{prop:GeomxGeom}, thus it suffices to establish the stability.

By \cite[Theorem~1.1]{FLZ}, every stability condition on $\Stab\Db(X_1\times X_2)$ is geometric.
Applying a Fourier--Mukai transform that sends a skyscraper sheaf to a slope‑stable vector bundle reduces the problem to the following claim:

\begin{center}
``If each $E_i$ is a slope-stable vector bundle on the elliptic curve $X_i$, then $E_1\boxtimes E_2$\\ is stable with respect to \emph{any} stability condition on $\Stab\Db(X_1\times X_2)$."
\end{center}

Slope‑stable vector bundles on an elliptic curve are simple and semihomogeneous. The external tensor product of two simple semihomogeneous bundles is again simple and semihomogeneous, and by \cite[Corollary~2.16]{FLZ} any such object is stable with respect to every stability condition on the product. This proves the claim.

\medskip\noindent\textbf{Uniqueness.}
The uniqueness of such product-type stability conditions again follows from the fact that geometric stability conditions are uniquely determined by the central charges and the phases of skyscraper sheaves.

\medskip\noindent\textbf{Part~\ref{item:ExE-item-a}}.
Suppose $E_1\boxtimes E_2$ is $(\sigma_1\boxtimes\sigma_2)$-stable.
Shifting if necessary, we may assume $E_i\in\Coh(X_i)$ and is indecomposable, hence slope‑semistable.
Assume for contradiction that some $E_i$ is not slope‑stable.
Let $\{F^{(1)}_i\}$ be the slope-stable factors of $E_1$ (all of phase $\phi_1$) and $\{F^{(2)}_j\}$ be the slope-stable factors of $E_2$ (all of phase $\phi_2$).
Then, each $F^{(1)}_i\boxtimes F^{(2)}_j$ is $(\sigma_1\boxtimes\sigma_2)$-stable of phase $\phi_1+\phi_2$. 
Since $E_1\boxtimes E_2$ can be obtained from these factors by a sequence of extensions, it would be strictly semistable of the same phase, contradicting the assumed stability.
Hence both $E_1$ and $E_2$ must be slope-stable.

\medskip\noindent\textbf{Part~\ref{item:ExE-item-b}}.
The argument for part (b) is identical to that given for Theorem~\ref{thm:P1xP1}.
\end{proof}

\appendix
\section{A fact on pure algebraic stability conditions.}
\begin{prop}\label{prop:appendix-pure}
Let $\cE=\{E_1,\ldots,E_n\}$ be a full strong exceptional collection in a triangulated category $\cD$.
Given positive numbers $m_1,\ldots,m_n>0$ and real numbers $\phi_1<\cdots<\phi_n$ with $\phi_i+1\leq \phi_{i+1}$ for all $i$, there exists a unique stability condition on $\cD$ such that:
\begin{enumerate}[label=(\roman*)]
    \item Each $E_i$ is stable of phase $\phi_i$.
    \item $Z(E_i)=m_ie^{\sqrt{-1}\pi\phi_i}$.
\end{enumerate}
Moreover, with respect to this stability condition, the only stable objects are $E_i$'s and their shifts.
\end{prop}
\begin{proof}
Define $p_i=-\lceil\phi_i-1\rceil$; then $ \phi_i + p_i \in (0, 1]$ for every $i$. 
For any $i<j$, we have $\phi_i+1\leq\phi_{i+1}\leq\phi_j$, therefore $p_i>p_j$.
Hence the shifted collection
$$
\cE\coloneqq\{E_1[p_1],\ldots,E_n[p_n]\}
$$
is a complete Ext-exceptional collection. 
By \cite[Section~3.3]{Macri} (see Lemma~\ref{lemma:Macri}), there exists a unique stability condition on $\cD$ for which each $E_i[p_i]$ is stable of phase $ \phi_i + p_i$, and has central charge $Z(E_i[p_i])=m_ie^{\sqrt{-1}\pi(\phi_i+p_i)}$. 
Shifting back yields the required stability condition for the original objects $E_i$ , proving the existence and uniqueness.

It remains to show that the only stable objects are the $E_i$'s and their shifts.
The heart associated to this stability condition is $\cA=\left<E_1[p_1],\ldots,E_n[p_n]\right>$.
Every object $E\in\cA$ admits a finite filtration in $\cA$ whose subquotients are isomorphic to some $E_i[p_i]$.
We claim that 
$$
\Hom^1(E_j[p_j],E_i[p_i])=0 \quad \text{ if } \quad \phi(E_i[p_i])<\phi(E_j[p_j]).
$$
Indeed, the extension is nonzero only if $i>j$ and $p_j=p_i+1$; in this case,
$$
\phi(E_i[p_i])=\phi_i+p_i\geq(\phi_j+1)+(p_j-1)=\phi(E_j[p_j]),
$$
contradicting the assumed inequality. Hence the claim holds.

Because these extensions vanish, we may reorder any such filtration so that the subquotients appear with non‑increasing phases.
Therefore, a non‑zero object of $\cA$ is stable exactly when its filtration has length one, i.e.~when it is isomorphic to some $E_i[p_i]$. Thus $E_i$'s and their shifts are the only stable objects.
\end{proof}

\bigskip
\bibliography{ref}
\bibliographystyle{alpha}

\ \\

\small{
\noindent Yu-Wei Fan \\
\textsc{Center for Mathematics and Interdisciplinary Sciences, Fudan University, \\ Shanghai 200433, China}\\
\textsc{Shanghai Institute for Mathematics and Interdisciplinary Sciences (SIMIS), \\  Shanghai 200433, China}\\
\texttt{yuweifanx@gmail.com}
}

\end{document}